\documentclass[10pt,twoside,a4paper]{article}

\usepackage[utf8]{inputenc}

\usepackage[T1]{fontenc}

\usepackage{amsmath, bm}

\usepackage{amssymb}

\usepackage{microtype}

\usepackage[backend=biber,maxnames=10,backref=true,hyperref=true,safeinputenc]{biblatex}

\DefineBibliographyStrings{english}{%
	backrefpage = {cited on page},
	backrefpages = {cited on pages},
}

\DeclareFieldFormat[report]{title}{``#1''}
\DeclareFieldFormat[book]{title}{``#1''}
\AtEveryBibitem{\clearfield{url}}
\AtEveryBibitem{\clearfield{note}}

\usepackage{graphicx}
\graphicspath{{images/}}

\usepackage{lmodern}
\usepackage{mathtools}
\usepackage{dsfont,bbm}
\usepackage{rotating}
\usepackage{placeins}
\usepackage{float}
\usepackage{enumerate}
\usepackage{multirow}
\usepackage{booktabs}
\usepackage{chngcntr}

\usepackage{graphicx,tikz,pgfplots}
\usepackage{tikz-3dplot}
\usetikzlibrary{arrows,shapes,fadings,intersections,decorations.pathreplacing,positioning,patterns}
\usetikzlibrary{calc, arrows.meta}
\usetikzlibrary{shapes.geometric}
\graphicspath{{images/}}
\pgfplotsset{compat=newest}

\usetikzlibrary{quotes}
\usetikzlibrary{angles}
\usetikzlibrary{babel}
\usetikzlibrary{decorations.pathreplacing}
\usetikzlibrary{decorations.markings}
\usepackage{mathabx}

\usepackage[english]{babel}

\title{Raster Scan Diffraction Tomography} 

\author{Peter Elbau$^{1}$\\
	{\footnotesize\href{mailto:email}{peter.elbau@univie.ac.at}}
	\and
	Noemi Naujoks$^{1,3}$\\
	{\footnotesize\href{mailto:email}{noemi.naujoks@univie.ac.at}}
	\and Otmar Scherzer$^{1,2,3}$\\	
	{\footnotesize\href{mailto:email}{otmar.scherzer@univie.ac.at}}
}

\date{\today}

\usepackage[pdftex,colorlinks=true,linkcolor=blue,citecolor=green!50!black,urlcolor=blue,bookmarks=true,bookmarksnumbered=true,pdfusetitle]{hyperref}
\hypersetup
{
	pdfauthor={P. Elbau, N. Naujoks, O. Scherzer},
	pdftitle={Raster Scan Diffraction Tomography}, 
}

\usepackage[hyperref,amsmath,thmmarks]{ntheorem}
\usepackage{aliascnt}

\newaliascnt{proposition}{lemma}

\aliascntresetthe{proposition}

\newaliascnt{corollary}{lemma}
\newtheorem{corollary}[corollary]{Corollary}
\aliascntresetthe{corollary}

\newaliascnt{theorem}{lemma}
\newtheorem{theorem}[theorem]{Theorem}
\aliascntresetthe{theorem}

\theorembodyfont{\normalfont}
\newaliascnt{definition}{lemma}
\newtheorem{definition}[definition]{Definition}
\aliascntresetthe{definition}

\newaliascnt{assumption}{lemma}

\aliascntresetthe{assumption}

\newaliascnt{remark}{lemma}
\newtheorem{remark}[remark]{Remark}
\aliascntresetthe{remark}

\newaliascnt{example}{lemma}
\newtheorem{example}[example]{Example}
\aliascntresetthe{example}

\theoremstyle{nonumberplain}
\theoremseparator{:}
\theoremheaderfont{\normalfont\itshape}

\theoremsymbol{\ensuremath{\square}}
\newtheorem{proof}{Proof}

\usepackage[a4paper,centering,bindingoffset=0cm,marginpar=2cm,margin=2.5cm]{geometry}

\usepackage[pagestyles]{titlesec}
\titleformat{\section}[block]{\large\sc\filcenter}{\thesection.}{0.5ex}{}[]
\titleformat{\subsection}[runin]{\bf}{\thesubsection.}{0.5ex}{}[.]

\newpagestyle{headers}

{
	
	\headrule
	
	\sethead[\footnotesize\thepage][\footnotesize\sc  P. Elbau, N. Naujoks, O. Scherzer][]{}{\footnotesize\sc Raster Scan Diffraction Tomography}{\footnotesize\thepage}
	
	\setfoot{}{}{}
	
}

\usepackage[font=footnotesize,format=plain,labelfont=sc,textfont=sl,width=0.75\textwidth,labelsep=period]{caption}
\usepackage[labelfont = normalfont]{subcaption}

\usepackage{dsfont}
\usepackage{bbm}
\newcommand{\N}{\mathds{N}}

\newcommand{\R}{\mathds{R}}
\newcommand{\C}{\mathds{C}}
\newcommand{\Sp}{\mathbb{S}}
\newcommand{\usc}{u^{\mathrm{sca}}}
\newcommand{\ui}{u^{\mathrm{inc}}}
\newcommand{\ut}{u^{\mathrm{tot}}}

\definecolor{myorange}{RGB}{246,121,48}
\definecolor{flame}{rgb}{0.75,0.21,0.17}
\definecolor{mygreen}{RGB}{46,139,87}

\let\RE\Re
\let\Re=\undefined
\DeclareMathOperator{\Re}{\RE e}
\let\IM\Im
\let\Im=\undefined
\DeclareMathOperator{\Im}{\IM m}
\DeclareMathOperator{\supp}{supp}

\newcommand{\norm}[1]{\left\|#1\right\|}
\newcommand{\set}[1]{\left\{#1\right\}}
\newcommand{\inner}[2]{\left<#1,#2\right>}

\begin{document}
\def\sectionautorefname{Section}
\def\subsectionautorefname{Section}
	
	\maketitle
	\thispagestyle{empty}
	\begin{center}
		\hspace*{5em}
		\parbox[t]{12em}{\footnotesize
			\hspace*{-1ex}$^1$Faculty of Mathematics\\
			University of Vienna\\
			Oskar-Morgenstern-Platz 1\\
			A-1090 Vienna, Austria}
		\hfil
		\parbox[t]{17em}{\footnotesize
			\hspace*{-1ex}$^2$Johann Radon Institute for Computational\\
			\hspace*{1em}and Applied Mathematics (RICAM)\\
			Altenbergerstraße 69\\
			A-4040 Linz, Austria}\\
		\vspace*{0.5cm}
		\hspace*{5em}
		\parbox[t]{18em}{\footnotesize
			\hspace*{-1ex}$^3$Christian Doppler Laboratory for Mathematical\\
			\hspace*{1em}Modelling and Simulation of Next Generation\\
			\hspace*{1em}Medical Ultrasound Devices (MaMSi)\\
			Oskar-Morgenstern-Platz 1\\
			A-1090 Vienna, Austria}
	\end{center}

	\begin{abstract}
	Diffraction tomography is a widely used inverse scattering technique for quantitative imaging of weakly scattering media. In its conventional formulation, diffraction tomography assumes monochromatic plane wave illumination. This assumption, however, represents a simplification that often fails to reflect practical imaging systems such as medical ultrasound, where focused beams are used to scan a region of interest of the human body. Such measurement setups, combining focused illumination with scanning, have not yet been incorporated into the diffraction tomography framework. To bridge this gap, we extend diffraction tomography by modeling incident fields as Herglotz waves, thereby incorporating focused beams into the theory. Within this setting, we derive a new Fourier diffraction relation, which forms the basis for quantitative tomographic reconstruction from scanning data. Using this result, we systematically analyze how different scan geometries influence the reconstruction.
	\end{abstract}

	\section{Introduction}\label{sec:intro}
Ultrasound imaging is one of the most widely used medical imaging modalities, providing non-invasive, real-time visualization of internal body structures. It plays a crucial role in diagnostics, guiding treatments and monitoring patient health. Its safety, accessibility, and cost-effectiveness make it an essential tool in modern healthcare.

In conventional ultrasound imaging, a series of focused acoustic beams is transmitted into tissue, and the returning echoes are received and processed to form an image \cite{Hos19}. While highly effective for visualizing anatomical structures, this method yields primarily qualitative information about tissue structure, rather than quantitative data about its physical properties. Quantitative ultrasound imaging, in contrast, seeks to reconstruct tissue acoustic parameters from the measured signals. This approach allows for objective tissue characterization and enhances the ability to distinguish between healthy and pathological regions, making it highly relevant for clinical applications.

The pursuit of quantitative imaging has long motivated research in ultrasound, dating back to the development of ultrasound computed tomography in the 1970s \cite{Gre77}. Within this context, diffraction tomography has become a well-established approach \cite{Nat15, NorLin79}. Here, a weakly scattering object is illuminated by an incident wave, and information about its internal structure is inferred from measurements of the scattered field. A major advantage of diffraction tomography is its computational efficiency. By employing the first-order Born (or Rytov) approximation, the non-linear inverse scattering problem is simplified to a linear one. The resulting Fourier diffraction theorem \cite{Wol69} establishes a direct link between the Fourier transform of the measured data and the object’s scattering potential, enabling explicit reconstruction via Fourier inversion, commonly referred to as \emph{filtered backpropagation} \cite{Dev82}.

However, classical diffraction tomography relies on idealized assumptions regarding both the incident field and the measurement geometry. Here, the object is illuminated by monochromatic plane waves arriving from a broad range of directions \cite{KakSla01, Nat15, Wol69}. While this approach works well for certain imaging scenarios where the object can be surrounded by transmitters and receivers, this setting does not reflect clinical practice in ultrasound imaging, where beams are typically emitted only from one side of the body, and, crucially, these beams are focused. Focusing concentrates acoustic energy at the focal region and thereby enables the visualization of deeper tissue layers with sufficient resolution \cite{Kut91}. This stands in sharp contrast to the plane-wave, full-angle illumination assumed in the classical framework. Recent work \cite{KirNauSchYan24} extended the classical diffraction tomography framework to account for focused beams by modeling them as superpositions of plane waves while still assuming full-angle data.

In this work, we go one step further by considering a series of focused beams that are actively scanned across the region of interest. Moreover, recent advances in hardware development, such as multi-aperture \cite{HalHooSamSchwLop24} and flexible transducer \cite{NeePetVerPeeHaa24}, are expanding what is experimentally possible. These devices could, in principle, emit, scan, and record signals from a much broader range of angles and focal positions than conventional probes. To fully exploit these capabilities and to bridge the gap between classical diffraction tomography and practical ultrasound systems, a theoretical framework is required that can accommodate arbitrary scanning geometries.

Therefore, we extend the diffraction tomography framework to a general scanning setup, see \autoref{fig:scangeo}, which combines focused illumination with active scanning.  We consider an object in a general Euclidean space~$\R^d$, $d\ge2$, probed by a focused beam propagating in a prescribed direction $\omega\in \Sp^{d-1}\coloneqq\{{x}\in\R^d \mid \norm{{x}} = 1\}$. The incident beam is hereby shifted such that the focal point moves on a predefined hyperplane orthogonal to a chosen direction $\nu\in\Sp^{d-1}$. For each position of the focal point, the interaction with the object generates a scattered wave field, which is subsequently measured at every point on the plane $\R^{d-1}\times\{L\}$.
We call this setup \emph{Raster Scan Diffraction Tomography}.

The parameters $\omega$ and $\nu$ in our formulation can be independently chosen, allowing for a wide range of scanning arrangements from standard linear probes to more exotic configurations. We derive a new Fourier diffraction theorem for such a scan geometry, and analyze which Fourier coefficients of the scattering potential of the object are accessible from the measurements. By systematically characterizing how different choices of beam direction $\omega$ and scan normal $\nu$ influence the recoverable parts of the object’s Fourier spectrum, we provide new insights into how the scanning geometry governs the object information and which scanning geometries are optimal for practical quantitative reconstructions.

\subsection*{Outline}
In \autoref{sec:forward}, we introduce the mathematical model describing focused beam scanning and its resulting measurement process. Based on this formulation, we derive a new Fourier diffraction theorem that applies to this scanning measurements. In \autoref{sec:Fourier_recon}, we discuss the reconstruction of the scattering potential via Fourier inversion. Finally, \autoref{sec:2Dcoverages} illustrates the Fourier space coverages for different scan geometries in two dimensions.

\begin{figure}[t]
	\centering
		\begin{subfigure}{0.3\textwidth}
		\vspace*{2pt}
		\includegraphics[scale = 0.78]{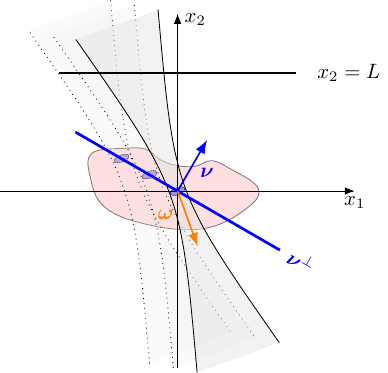}
		\caption{General scan geometry}
		\label{subfig:genset}
	\end{subfigure}\hfill
	\begin{subfigure}{0.3\textwidth}
		\includegraphics[scale = 0.78]{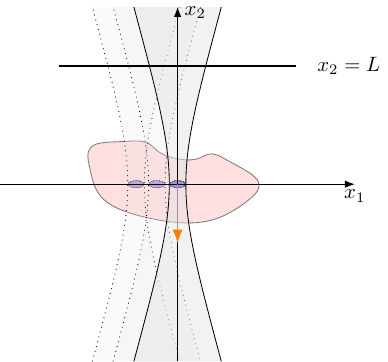}
		\caption{Reflection imaging with $\nu=\omega=-e_2$}
		\label{subfig:RI}
	\end{subfigure}\hfill
	\begin{subfigure}{0.3\textwidth}
		\includegraphics[scale = 0.78]{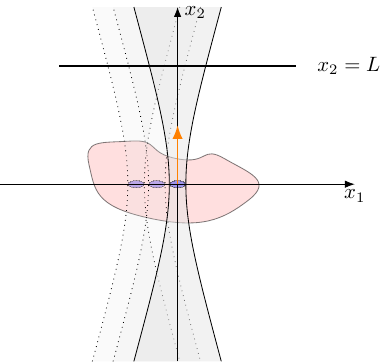}
		\caption{Transmission imaging with $\nu=\omega=e_2$}
		\label{subfig:TI}
	\end{subfigure}
	\caption{Schematic overview of the general scan geometry in $d=2$. (a) A focused beam propagates in the direction $\omega\in\mathbb{S}^{1}$ and scans an object by shifting the focal point along the line $\nu^\perp$ orthogonal to the direction $\nu\in\mathbb{S}^1$. The resulting scattered waves are measured on a receiver line $x_2 = L$ outside the object. (b) Reflection imaging with the beam pointing downward,  $\omega=\nu=-e_2$. (c) Transmission imaging with the beam pointing upward, $\omega=\nu=-e_2$. }
	\label{fig:scangeo}
\end{figure}

	\section{The forward model}\label{sec:forward}
We first establish the mathematical model of wave propagation and scattering that will serve as the basis for our analysis. Let $d\in\N \setminus \{1\}$ denote the space dimension, where the main practical interest is, of course, the case $d=3$. We assume that the object to be reconstructed is contained inside a ball 
\[
\mathcal B^d_{r} \coloneqq \{{x}\in\R^d\mid\norm{{x}}<r\}
\]
of some radius $r>0$ where $\norm{\cdot}$ denotes the Euclidean norm. The object is characterized by the spatially varying acoustic refractive index $n \colon \R^d \to \C$, which satisfies $n( x) = 1$ for $ x \in \R^d\setminus\mathcal B_{r}^d$.

Let $\ui\colon \R^d \to \C$ represent the given incident field, which we assume to be a solution of the Helmholtz equation
\begin{equation}\label{eq:Helm}
	\Delta \ui(x) + k_0^2\ui(x) = 0, \qquad x\in\R^d, 
\end{equation}
where $k_0>0$ is the wave number. Following \cite[section 8.2]{ColKre19} the total field $\ut = \ui+\usc$ satisfies
\[
\Delta \ut(x) +k_0^2n(x)\ut(x) = 0, \qquad \ x\in\R^d,
\]
where the scattering part $\usc$ satisfies the Sommerfeld radiation condition \cite{ColKre19,Som12}
\[
\lim\limits_{\norm{{x}}\to \infty}\norm{{x}}^{\frac{d-1}{2}}\left(\frac 1 {\norm{ x}}  \inner{x}{ \nabla \usc(x)} -ik_0\usc(x)\right) = 0,
\]
ensuring that only outgoing waves are considered.

\subsection{First-order scattering model}
We introduce the scattering potential
\[
f({x})=k_0^2[n({x})^2-1], \qquad {x}\in\R^d,
\]
whose support is by construction contained in $\overline{\mathcal B^d_{r}}\subseteq[-r,r]^d$. In diffraction tomography, it is common to assume a weakly scattering object, meaning that $f$ is sufficiently small such that the scattered field can be approximated by the single-scattered field, that is the first Born approximation \cite{Wol69}
\begin{equation}\label{eq:usc}
	\usc({x}) \approx u({x}) \coloneqq \int_{\R^d}G({x}- x')f({x}')\ui({x}')d x', \quad {x}\in\R^d,
\end{equation}
where 
\[
G({x})\coloneqq \frac{i}{4}\left(\frac{k_0}{2\pi\norm{{x}}}\right)^{\frac{d-2}{2}} H_{\frac{d-2}{2}}^{(1)}\left(k_0\norm{{x}}\right)
\]
is the outgoing fundamental solution to the Helmholtz equation in $\R^d$, see \cite{Agm90}. Here, $H_\rho^{(1)}$ denotes the Hankel function of the first kind of order $\rho$. Specifically, in two and three dimensions, we have that
\[
G(x) =
\begin{cases}
	\displaystyle \frac{i}{4}\, H_0^{(1)}\left( k_0 \norm{x}\right) , & d=2, \\[2mm]
	\displaystyle \frac{e^{i k_0  \norm{x}}}{4 \pi \norm{x}}, & d=3.
\end{cases}
\]

The dependence of $\usc$ on the scattering potential $f$ is non-linear. The non-linear inverse problem of recovering $f$ from measurements of $\usc$ has been studied extensively in the literature, see, for example,  \cite{CakColHad16, ColKir96, ColKirPai89,ColMon88,GutKli93,Hoh01}; a comprehensive overview is provided in \cite[Section 11]{ColKre19}. A conceptually different, data-driven approach tailored to reflection imaging setups was proposed in \cite{BorDruMamZas18, BorDruMamZasZim20, BorGarMamZim22}.

However, the map from $f$ to the single-scattered field $u$ is linear so that it simplifies the inversion considerably if we assume to have measurement data for $u$.

\subsection{Herglotz wave illumination}
We model an incident beam satisfying \autoref{eq:Helm} as a superposition of plane waves propagating in different directions. Specifically, let 
\[
\Sp^{d-1}_{k_0} \coloneqq \left\lbrace x \in \R^d \mid \norm{x} = k_0\right\rbrace 
\] 
denote the sphere of radius $k_0$ in $\R^d$. We write $L^2(\Sp_{k_0}^{d-1})$ for the Lebesgue space of all square-integrable functions on this sphere. The incident beam is then represented by a Herglotz wave function \cite[Definition 3.26]{ColKre19} of the form
\begin{equation}\label{eq:ui}
	\ui(x) = \int_{\Sp^{d-1}_{k_0}} a(s) \, e^{i \inner{x }{s}} \, dS(s),
\end{equation}
where the Herglotz density $a \in L^2(\Sp^{d-1}_{k_0})$ represents the amplitude of each individual plane wave. 

Our main result, stated in \autoref{thm:FDT}, applies to arbitrary Herglotz densities $a\in L^2(\Sp_{k_0}^{d-1})$ without further assumptions. However, in the setting of \autoref{subfig:genset} the beam propagates in a fixed direction $\omega\in \Sp^{d-1}$. To account for this, we restrict the Herglotz density to the corresponding half-sphere. More precisely, we impose 
\begin{equation}\label{eq:halfsupport}
	a(s) = 0 \quad \text{whenever } \inner{s}{\omega} \leq 0.
\end{equation}

The standard example we have in mind is a Gaussian beam.

\begin{example}[Gaussian beam]\label{ex:beam}
	A focused beam propagating in direction $\omega\in \Sp^{d-1}$ can be obtained by prescribing the Herglotz density as a Gaussian function. Specifically, for $s\in \Sp^{d-1}_{k_0}$, we set
	\begin{equation}\label{eq:gaussian}
		a(s) \coloneqq \begin{cases}
		e^{-A\norm{\tilde{s}}^2}, \quad &\inner{s}{\omega} >0,
		\\ 0, & \inner{s}{\omega}\leq 0, 
		\end{cases}
	\end{equation}
	where $\tilde{s} \coloneqq s-\inner{s}{\omega}\omega$ is the component of $s$ orthogonal to $\omega$ and the parameter $A>0$ controls the beam waist, determining how strongly the beam is focused.
	A beam described by \autoref{eq:ui} with the density from \autoref{eq:gaussian} has its focal point located in the origin and is consistent with the modeling in \cite{AgrPat79}.
\end{example}

Inserting \autoref{eq:ui} in \autoref{eq:usc} and changing the order of integration, we obtain the scattered wave
\begin{align}\label{eq:u_hgw}
	u( x) = \int_{\Sp_{k_0}^{d-1}}a({s})w({x},{s})dS(s), 
\end{align}
where the function
\begin{equation}\label{eq:w}
		w({x}, s) = \int_{\R^{d}}G({x}-{x}')f({x}')e^{i\langle x',s\rangle}d{x}'
\end{equation}
represents the scattered wave induced by the incident plane wave $e^{i\langle x',s\rangle}$ propagating in direction $s\in\Sp^{d-1}_{k_0}$. Hence, the scattered field generated by a Herglotz wave is a superposition of the scattered waves $w$ from the different incident plane waves weighted by the Herglotz density $a\in L^2(\Sp^{d-1}_{k_0})$.

\subsection{Scanning measurements} 
The object is scanned by translating the incident beam along the hyperplane 
\[
\nu^\perp \coloneqq \left\lbrace y\in\R^d\mid \inner{y}{\nu} = 0 \right\rbrace 
\] 
which is oriented by the normal vector $\nu\in \Sp^{d-1}$. We refer to this hyperplane as the \emph{scan plane}. The scanning process is modeled by applying a translation to $\ui$, that is, we consider for every shift $y\in\nu^\perp$ an incident wave $\ui_y$ of the form
\begin{equation}\label{eq:ui_rigidm}
	\ui_{{y}}( x)\coloneqq\ui( x-{y})  = \int_{\Sp^{d-1}_{k_0}}a(s)e^{i\inner{x}{s}}e^{-i\inner{ y }{s}}dS({s}),\qquad  x\in\R^{d},\ y\in \nu^\perp.
\end{equation}
Herein, we assume that the orientation $\nu$ of the focal plane remains fixed during the experiment. The beam direction $\omega$ is encoded in the support of the density $a$ as specified in \autoref{eq:halfsupport} and is also fixed.

In this article, we distinguish between three \emph{scan regimes} depending on the relative orientation of the scan normal $\nu\in \Sp^{d-1}$ and the beam direction $\omega\in \Sp^{d-1}$; see also the geometric setup in \autoref{subfig:genset}.
\begin{enumerate}
	\item \textbf{Perpendicular scan:} in this configuration, the scan plane is orthogonal to the beam's propagation direction, i.e., $\nu = \omega$. This means that the focal point moves within this plane, always maintaining a fixed distance from the source along the beam axis.
	This setting reflects standard B-mode ultrasound, where lateral motion of the beam yields cross-sectional images at a constant depth.
	\item \textbf{Parallel scan:} One translation direction is along the beam direction, that is $\omega\in\nu^\perp$. This implies that the beam axis itself lies within the scan plane $\nu^\perp$. The focal point is still constrained to move within this plane, but this movement now has two components: a lateral shift (in $d-2$ dimensions) and a shift along the beam axis. This geometry therefore models a scan where the focal depth is actively varied, corresponding to a depth-wise scan of the object.

	\item \textbf{Tilted scan:} all intermediate cases, where neither perpendicular nor parallel alignment holds, i.e., $\nu\neq\omega$ and $\nu\not\perp\omega$. Physically, this corresponds to a scan where the focal point moves along an oblique plane relative to the beam axis. Such configurations may occur in practice when, for example, an ultrasound probe is slightly tilted relative to the imaging plane.
\end{enumerate}

Utilizing \autoref{eq:usc} with the incident wave $\ui_y$ and changing the order of integration, the corresponding scattered waves in Born approximation, which we will denote by $u_{{y}}$, are given by 
\begin{equation}\label{eq:uy}
	u_{{y}}({x}) = \int_{\R^d}G({x}-{x}')f({x}')\ui_{{y}}({x}')d{x}'=\int_{\Sp_{k_0}^{d-1}}a({s})w({x},{s})e^{-i\inner{y}{s}}dS(s), 
\end{equation}
where $w(x,s)$ is defined as in \autoref{eq:w}.
 
In our experiment, we assume that the scattered waves are detected at every point on a hyperplane $e_d^\perp+Le_d$ orthogonal to the $d$th standard basis vector $e_d$ at a distance $L>r$ from the origin, which ensures that it is outside the object. More precisely, we model the given scanning measurements as
\begin{equation}\label{eq:meas}
	m(x,y) \coloneqq u_{y}(x),\qquad x\in e_d^\perp+Le_d,\ y\in \nu^\perp.
\end{equation}
As our motivating application lies in acoustical imaging, we assume that these measurements are complex-valued, that is, both amplitude and phase information of the scattered waves are available. If only intensity (phaseless) data can be measured, additional phase retrieval techniques are required; see, for instance, \cite{AmmChoZou16, Nov15, KliRom16, HohNovSiv24}.

Depending on the relative orientation of the propagation direction $\omega\in \Sp^{d-1}$ of the beam and the measurement plane $e_d^\perp+Le_d$, different Fourier coefficients of the scattering potential influence the measurements, and we want to distinguish in this context three \emph{imaging regimes}.
	\begin{enumerate}
		\item \textbf{Reflection imaging:} the beam propagates away from the detector, i.e., $\omega= -e_d$. This corresponds to standard ultrasound imaging, where the recorded signal consists of reflected wave components; see \autoref{subfig:RI}.
		\item \textbf{Transmission imaging:} the beam propagates in the opposite direction, i.e., $\omega= e_d$, so that the detector captures only the transmitted wave components, see \autoref{subfig:TI}.
		\item \textbf{Oblique imaging:} the beam propagates in an arbitrary direction $\omega\in \Sp^{d-1}\setminus\{\pm e_d\}$ so that the detector can capture a mixture of transmitted and reflected components of the scattered waves.
	\end{enumerate}

\subsection{Fourier diffraction theorem for scanning measurements}\label{sec:FDT}
The Fourier diffraction theorem, originally formulated in \cite{Wol69}, forms the theoretical foundation for reconstructing the scattering potential $f$ from measurements of the scattered waves. It establishes a relation between the measured scattered wave from \autoref{eq:meas} and the spatial Fourier transform of the scattering potential. We now derive a version of the Fourier diffraction theorem adapted to the scanning geometry.

For this purpose, we will apply the Fourier transform to the measurement data $m$ with respect to the detection point $x$ on the measurement plane $e_d^\perp+Le_d$ and with respect to the shift $y$ of the incident plane along the scan plane $\nu^\perp$, which may in general be arbitrarily tilted relative to the measurement plane.

To be specific, we adopt (on the space $C_c^\infty(\R^d)$ of smooth functions with compact support where we have no problems with the existence of the integrals) the convention 
\begin{align*}
	\mathcal{F}_d\phi({k}) &=(2\pi)^{-\frac{d}{2}} \int_{\R^d}\phi({x})e^{-i\inner{k}{x}}d{x}, \qquad \phi \in C_c^\infty(\R^d), \\
	\mathcal{F}^{-1}_d\hat\phi({x}) &= (2\pi)^{-\frac{d}{2}} \int_{\R^d}\hat\phi({k})e^{i\inner{k}{x}}d{k}, \qquad \hat\phi \in C_c^\infty(\R^d),
\end{align*}
for the $d$-dimensional Fourier transform $\mathcal F_d\colon L^2(\R^d)\to L^2(\R^d)$ and its inverse $\mathcal F_d^{-1}$.

To take the Fourier transforms only along the $(d-1)$-dimensional planes $e_d^\perp+Le_d$ and $\nu^\perp$, we want to introduce for an arbitrary hyperplane $E\coloneqq v^\perp + t v\subseteq\R^d$, described by a unit normal vector $v\in\Sp^{d-1}$ and a shift $t\in\R$ along $v$, the Fourier transform $\mathcal F_E\colon L^2(E)\to L^2(v^\perp)$ and its inverse along $E$ via
\begin{equation}\label{eq:fftnu}
	\begin{aligned}
		\mathcal F_{E} \phi(\xi)
		&= (2\pi)^{-\frac{(d-1)}{2}}\int_{E} \phi(y)\, e^{- i\inner{\xi}{y}}  dS(y), \ \qquad \xi\in v^\perp,\ \phi\in C_c^\infty(E), \\
		\mathcal F^{-1}_{E}\hat\phi(y)
		&= (2\pi)^{-\frac{(d-1)}{2}}\int_{v^\perp}\hat\phi(\xi)\, e^{i\inner{\xi}{y}} dS(\xi), \qquad y\in E, \ \hat\phi \in C_c^\infty(v^\perp).
	\end{aligned}
\end{equation}

To verify the consistency of the definition, we choose a rotation $T$ with $T(\R^{d-1}\times\{0\}) = v^\perp$.
Every $y\in E$ can then be written uniquely as 
\[
y = T(\bar p,0) + t v \qquad\text{for some } \bar p\in\R^{d-1},
\]
and every frequency variable $\xi\in v^\perp$ can be parametrized by 
\[\xi = T(\bar q,0)\qquad\text{for some }\bar q\in\R^{d-1}.
\]
Since rotations and translations preserve the surface measure and since it follows from $\xi\in v^\perp$ that $\inner{y}{\xi}=\inner{\bar p}{\bar q}$, we have that
\begin{align*}
\mathcal F_E\phi(T(\bar q,0))&=(2\pi)^{-\frac{d-1}{2}}\int_{\R^{d-1}}\phi\big(T(\bar p,0)+t v\big)e^{-i\langle\bar q,\bar p\rangle}\,d\bar p, \\
\mathcal F_E^{-1}\hat\phi(T(\bar p,0)+tv)&=(2\pi)^{-\frac{d-1}{2}}\int_{\R^{d-1}}\hat\phi\big(T(\bar q,0)\big)e^{i\langle\bar q,\bar p\rangle}\,d\bar q.
\end{align*}
The inversion theorem for the $(d-1)$-dimensional Fourier transform thus shows us that $\mathcal F_E^{-1}$ is indeed the inverse of $\mathcal F_E$.

We want to define now for the measurements $m\colon(e_d^\perp+Le_d)\times\nu^\perp\to\C$ from \autoref{eq:meas} the Fourier transform with respect to both components as
\begin{equation}\label{eq:FTmeas}
\mathcal F_{e_d^\perp+Le_d,\nu^\perp}m(k,\xi) \coloneqq\mathcal F_{e_d^\perp+Le_d}\Big[x\mapsto \mathcal F_{\nu^\perp}^{-1}\big[ y \mapsto m(x,y)\big](\xi)\Big](k),
\end{equation}
where we chose for convenience for the second component the inverse Fourier transform. However, we remark that the function $m$ does not decay sufficiently fast at infinity to be square integrable. The Fourier transform in \autoref{eq:FTmeas} is thus to be interpreted as the Fourier transform of the regular tempered distribution given by $m$.

To express this Fourier transform, we define for every direction $v\in\Sp^{d-1}$ the half-space 
\begin{equation}\label{eq:H}
\Omega_{v}\coloneqq\left\lbrace {x}\in\R^d\mid \inner{x}{v} > 0 \right\rbrace 
\end{equation}
and introduce the hemisphere 
\begin{equation}\label{eq:Salpha}
	S_{v} \coloneqq \Sp^{d-1}_{k_0}\cap \Omega_{v}=\left\lbrace {x}\in\R^d \mid \norm{{x}}=k_0, \ \inner{x}{v} > 0 \right\rbrace, 
\end{equation}
which has radius $k_0$ and is oriented in the direction $v$. To parametrize this surface, we define the map
\begin{equation}\label{eq:para}
	{h}_{v}\colon \mathcal{B}_{k_0}^{d}\cap v^\perp\to S_{v}, \qquad  {h}_{v}(\xi) \coloneqq \xi +\kappa(\xi)v, 
\end{equation}
where $\kappa(\xi)\coloneqq \sqrt{k_0^2-\norm{\xi}^2}$. 

\begin{figure}
	\centering
	\begin{subfigure}{0.45\textwidth}
		\includegraphics{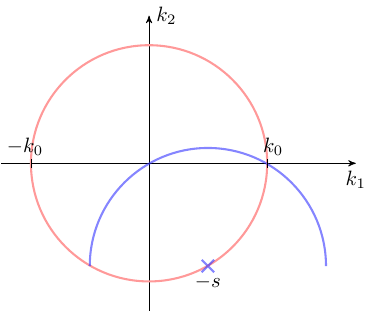}
		\caption{Semicircle parametrized by $k_1\mapsto h_{e_2}(k_1)-s$}
	\end{subfigure}\hfill
	\begin{subfigure}{0.45\textwidth}
		\includegraphics{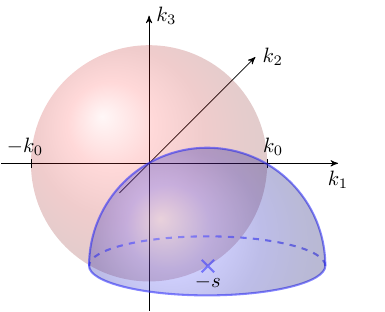}
		\caption{Hemisphere parametrized by $(k_1,k_2)\mapsto h_{e_3}(k_1,k_2)-s$}
	\end{subfigure}
	\caption{Illustration of the classical Fourier diffraction theorem stated in \autoref{eq:FDT}.  For plane-wave illumination with direction $s\in\mathbb{S}_{k_0}^{d-1}$, the Fourier transform of the scattering potential is accessible on a downward-oriented semicircle for $d=2$ or hemisphere for $d=3$, both of radius $k_0$ and centered at $-s\in\mathbb{S}_{k_0}^{d-1}$.} 
	\label{fig:hemisphere}
\end{figure}

\begin{theorem}[Fourier transform of the measurements]\label{thm:FDT}
	Let $d\in\N\setminus\{1\}$ and $f\in L^1(\R^d)$ be a function with $\supp(f)\subseteq \mathcal B^d_{r}$. Further, assume the incident wave is represented by \autoref{eq:ui_rigidm} with $a\in L^2(\Sp^{d-1}_{k_0})$. Choose an orientation $\nu \in \Sp^{d-1}$ of the scan hyperplane and a measurement distance $L>r$, see \autoref{subfig:genset}.
	
	Then, the Fourier transform of the measurement data $m$, as defined in \autoref{eq:FTmeas}, satisfies 
	\begin{align}\label{eq:PFDT}
		\mathcal F_{e_d^\perp+Le_d,\nu^\perp}m( k, \xi) =
		C({k},\xi)\Big(
		a\left({h}_{\nu}(\xi)\right)\mathcal F_d f\left( h_{e_d}({k})-{h}_{\nu}(\xi)\right)+ a\left({h}_{-\nu}(\xi)\right)\mathcal F_d f( h_{e_d}({k})-{h}_{-\nu}(\xi))\Big)
	\end{align}
	for $ k\in\mathcal B^{d}_{k_0}\cap e_d^\perp$ and $\xi\in\mathcal{B}^d_{k_0}\cap \nu^\perp$, where $h_{e_d}$ parametrizes the upper hemisphere and $h_{\pm\nu}$ parametrizes the hemisphere oriented along $\pm\nu$, see \autoref{eq:para}. The constant $C(k,\xi) $ is given by
	\begin{equation}\label{eq:const}
		C( k,\xi) \coloneqq (2\pi)^{\frac d2}\frac{ik_0e^{i\kappa({k})L}}{2\kappa({k})\kappa({\xi})}.
	\end{equation}
	Moreover, we get
	\begin{equation*}
		\mathcal F_{e_d^\perp+Le_d,\nu^\perp}m(k,\xi) = 0
	\end{equation*}
	for  $k\in\mathcal B^{d}_{k_0}\cap e_d^\perp$ and $\xi\in \nu^\perp\setminus\mathcal B^{d}_{k_0}$. 
\end{theorem}
\begin{proof}
	We begin by performing the partial Fourier transform of the measurement data defined in \autoref{eq:meas} with respect to the observation positions $x\in e_d^\perp+Le_d$, that is, we calculate
	\[
	\mathcal{F}_{e_d^\perp+Le_d}[x\mapsto m(x, y)](k) = \mathcal{F}_{e_d^\perp+Le_d}u_y(k) \qquad \text{for all } k\in e_d^\perp,
	\]
	where $u_y$ denotes the $y$-shifted Born approximation.
	Substituting the expression for $u_y$ from \autoref{eq:uy} and interchanging the order of integration gives
	\begin{equation}\label{eq:Ft1}
		\mathcal{F}_{e_d^\perp+Le_d}[x\mapsto m(x, y)](k) =\int_{\Sp^{d-1}_{k_0}}a(s)\mathcal{F}_{e_d^\perp+Le_d}[x\mapsto w(x, s)](k)e^{-i\inner{y}{s}}dS(s),
	\end{equation}
	where $w$ denotes the scattered wave induced by the plane wave $e^{i\inner{x}{s}}$, see \autoref{eq:w}. The interchange of the integrals in \autoref{eq:Ft1} is justified only in the distributional sense.
	Specifically, since the scattering potential $f\in L^1(\R^d)$ has compact support, it follows from \cite[Theorem~2.2]{KirQueSet25} that the corresponding field $w$ is a tempered distribution.
	Under these conditions, the same argument used in the proof of \cite[Theorem~2.6]{KirNauSchYan24} applies, ensuring that the exchange of integration and the Fourier transform is valid in the sense of tempered distributions.
	
	By the classical Fourier diffraction theorem, which in arbitrary spatial dimension is given in \cite[Corollary~4.1]{KirQueSet25}, one has for each fixed $s\in\Sp^{d-1}_{k_0}$ the relation
	\begin{equation}\label{eq:FDT}
		\mathcal{F}_{e_d^\perp+Le_d}[x\mapsto w(x, s)](k) = \sqrt{\frac{\pi}{2}}\frac{ie^{i\kappa({k})L}}{\kappa( k)}\mathcal{F}_df({h}_{e_d}({k})-{s}) \qquad \text{for all }{k}\in\mathcal B^{d}_{k_0}\cap e_d^\perp.
	\end{equation}
	 So only the Fourier coefficients of $f$ on the upper hemisphere centered at $s$, see \autoref{fig:hemisphere}, influence the Fourier transform of $w$ with respect to the first component. Substituting \autoref{eq:FDT} into \autoref{eq:Ft1} gives us
	\begin{equation}\label{eq:fft_focusspot}
		\mathcal{F}_{e_d^\perp+Le_d}[x\mapsto m(x, y)](k)=\sqrt{\frac{\pi}{2}}\frac{ie^{i\kappa(k)L}}{\kappa( k)}\int_{\Sp^{d-1}_{k_0}}a({s})\mathcal F_d f( h_{e_d}({k})-{s})e^{-i \inner{y}{s}}dS(s)
	\end{equation}	
	for all $ k\in\mathcal B^{d}_{k_0}\cap e_d^\perp$ and  ${y}\in \nu^\perp$.

	To further evaluate the integral in \autoref{eq:fft_focusspot}, we decompose the sphere into 
	\[
	\Sp^{d-1}_{k_0}= S_{\nu}\cup \overline{S_{-\nu}}= S_{\nu}\cup S_{-\nu}\cup \left({\Sp_{k_0}^{d-1}\cap\nu^\perp}\right),
	\]
	where $S_{\pm\nu}$ denotes the two hemispheres separated by the hyperplane $\nu^\perp$, see \autoref{eq:Salpha}. The boundary ${\Sp_{k_0}^{d-1}\cap\nu^\perp}$ between the hemispheres is a null set in $\Sp^{d-1}_{k_0}$ and does therefore not contribute to the integral in \autoref{eq:fft_focusspot}, which then becomes with the abbreviation $g({s}) \coloneqq a({s})\mathcal F_d f( h_{e_d}({k})-{s})$
	\begin{equation}
		\label{eq:fft_focusspot2}
		\int_{\Sp^{d-1}_{k_0}}g({s})e^{-i \inner{y}{s}}dS(s)=\int_{S_\nu}g({s})e^{-i \inner{y}{s}}dS(s)+\int_{S_{-\nu}}g({s})e^{-i \inner{y}{s}}dS(s).
	\end{equation}
	
	Next, we perform the substitution $s=h_{\pm\nu}(\xi)$, where $h_{\pm\nu}$ denotes the parametrization of each hemisphere according to \autoref{eq:para}.
	For $y\in\nu^\perp$ we have 
	\[
	\inner{y}{h_{\pm\nu}(\xi)}=\inner{y}{\xi}+\kappa(\xi)\inner{y}{\pm\nu} = \inner{y}{\xi},\quad \xi\in\mathcal{B}_{k_0}^d\cap\nu^\perp.
	\]
	
	Consider the standard parametrization $h\colon\mathcal{B}_{k_0}^{d-1}\to S_{e_d}$ of the upper hemisphere defined by $h(\bar \xi)\coloneqq(\bar\xi,\kappa(\bar\xi))$. Its surface element is given by the Gram determinant	\[
	\sqrt{\det(J_{{h}}(\bar{\xi})^\intercal J_{{h}}(\bar\xi))} =\frac{k_0}{\kappa(\bar\xi)}, \qquad \bar \xi\in\mathcal{B}_{k_0}^{d-1},
	\]
	where $J_{h}$ is the Jacobian of $h$. Since the parametrizations $h_{\pm\nu}$ differ from $h$ only by a rotation, and rotations preserve the surface measure, they induce the same surface element.
	
	Applying this change of variables in  \autoref{eq:fft_focusspot2}, we therefore obtain
	\begin{align*}
		\int_{\Sp^{d-1}_{k_0}}g({s})e^{-i\inner{y}{s}}dS({s})=k_0\int_{\mathcal{B}_{k_0}^{d}\cap  \nu^\perp}g\left(h_{\nu}(\xi)\right)\frac{e^{-i\inner{y}{\xi}}}{\kappa(\xi)}dS(\xi) +k_0\int_{\mathcal{B}_{k_0}^{d}\cap \nu^\perp}g(h_{-\nu}(\xi))\frac{e^{-i\inner{y}{\xi}}}{\kappa(\xi)}dS(\xi).
	\end{align*}
	Introducing $\mathbf{1}_M$ as the indicator function of a set $M$ and using the Fourier transform along $\nu^\perp$, defined in \autoref{eq:fftnu}, we have
	\[
	\int_{\mathcal{B}_{k_0}^{d}\cap\nu^\perp} 
	g(h_{\pm\nu}(x))\,\frac{e^{-i\inner{y}{\xi}}}{\kappa(x)}\,dS(\xi)
	= (2\pi)^{\frac{d-1}2}\mathcal{F}_{\nu^\perp}\Bigl[\xi \mapsto \mathbf{1}_{\{\|\xi\|<k_0\}}(\xi)\,\frac{g(h_{\pm\nu}(\xi))}{\kappa(\xi)} \Bigr](y) .
	\]

	Finally,  \autoref{eq:fft_focusspot} can be rewritten as 
	\begin{align}
		\mathcal{F}_{e_d^\perp+Le_d}[x\mapsto m(x, y)](k)
		=  (2\pi)^{\frac d2}\frac{ ik_0 e^{i\kappa(k)L}}{2\kappa(k)}
		\Biggl(
		\mathcal{F}_{\nu^\perp}&\Bigl[\xi \mapsto \mathbf{1}_{\{\|\xi\|<k_0\}}(\xi)\,\frac{g(h_{\nu}(\xi))}{\kappa(\xi)} \Bigr](y) \notag\\
		\qquad & +
		\mathcal{F}_{\nu^\perp}\Bigl[\xi \mapsto \mathbf{1}_{\{\|\xi\|<k_0\}}(\xi)\,\frac{g(h_{-\nu}(\xi))}{\kappa(\xi)} \Bigr](y)
		\Biggr)
	\end{align}
	and applying the inverse Fourier transform $\mathcal{F}_{\nu^\perp}^{-1}$ with respect to $y\in\nu^\perp$ then gives for $k\in \mathcal B^{d}_{k_0}\cap e_d^\perp$ and $\xi \in \nu^\perp$ the formula
	\[
	\mathcal F_{e_d^\perp+Le_d,\nu^\perp} m(k,\xi) =
	\begin{cases}
		C(k,\xi)\Bigl(g(h_\nu(\xi)) + g(h_{-\nu}(\xi))\Bigr), & \|\xi\| < k_0,\\
		0, & \|\xi\| \ge k_0,
	\end{cases}
	\]
	with the constant $C(k,\xi)$ as defined in \autoref{eq:const}. Thus, using the definition of $g$ completes the proof.
\end{proof}

The values of $\mathcal F_{e_d^\perp+Le_d,\nu^\perp}m(k,\xi)$ for $k\in e_d^\perp\setminus\mathcal B_{k_0}^d$ are generally not zero, but they decay exponentially as $L\to\infty$ and they are therefore in practice usually overshadowed by noise. So we do not want to use them for the reconstruction of $f$ from the measurements $m$.

\begin{figure}[t]
	\centering
	\begin{subfigure}[b]{0.45\textwidth}
	\includegraphics{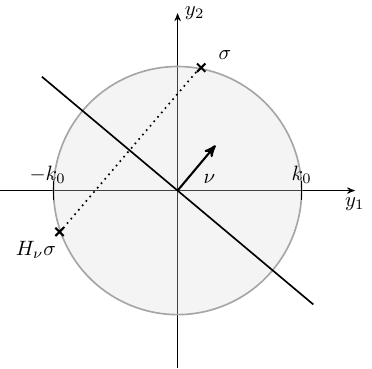}	
	\caption{$a(\sigma)\neq0$ for $\sigma\in \Sp^{d-1}_{k_0}$}
	\label{subfig:full}
	\end{subfigure}	\hfil
	\begin{subfigure}[b]{0.45\textwidth}
	\includegraphics{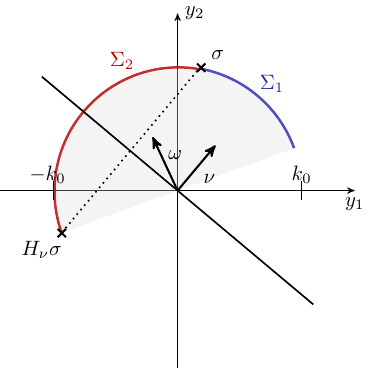}
	\caption{$a(\sigma)\neq0$ only for $\sigma\in S_{\omega}$}
	\label{subfig:half}
	\end{subfigure}	
	\caption{Illustration of the symmetry with respect to the scan plane $\nu^\perp$. (a) All directions $\sigma\in\Sp^{d-1}_{k_0}$ contribute to the Fourier diffraction relation when the beam profile $a\in L^{2}(\Sp^{d-1}_{k_0})$ is fully supported.
		(b) If $a(\sigma)\neq 0$ only for $\sigma\in S_{\omega}$, then only directions in $S_{\omega}$ contribute. Here, the subset $\Sigma_1$  contains directions whose reflections lie outside $ S_{\omega}$, whereas for $\Sigma_2$  the reflections remain inside.}
	\label{fig:symeta}
\end{figure}

\begin{remark}[Connection to Synthetic Aperture Diffraction Tomography]
	In the pure trans\-mission experiment, as illustrated in \autoref{subfig:TI}, we select $\nu = \omega = e_d$ and choose, as required in \autoref{eq:halfsupport}, a Herglotz density $a\in L^2(\Sp^{d-1}_{k_0})$ with $a(s) = 0$ for all $s\in S_{-e_d}$. Under these conditions, the coefficient $a(h_{-\nu}(\xi))$ of the second term on the right-hand side of \autoref{eq:PFDT} is zero, and the Fourier transform of the measurements simplifies to 
	\begin{align*}
		\mathcal F_{e_d^\perp +Le_d,e_d^\perp}m( k, \xi) =
		C({k},\xi)	a\left(h_{e_d}(\xi)\right)\mathcal F_d f\left( h_{e_d}({k})-h_{e_d}(\xi)\right),\qquad   k,\xi\in\mathcal B^{d}_{k_0}\cap e_d^\perp.
	\end{align*}
	This relation coincides with the Fourier diffraction formula known from \emph{synthetic aperture diffraction tomography}, as presented in \cite[eq. (28)]{NahPanKak84} for $d=2$. In their setup, a single-element transducer is moved along a line and is excited sequentially with measurements taken behind the object. The spherical waves by each transducer element are modeled via a superposition of plane waves with the same amplitude, which corresponds to a Herglotz wave with constant Herglotz density $a$. This is why our formula resembles the Fourier diffraction relation in this special case.
\end{remark}

For the subsequent analysis, we aim to rewrite \autoref{eq:PFDT} in a more compact form by noting that the expression on the right hand side depends rather on the points
\[
\eta \coloneqq h_{e_d}(k) \in S_{e_d} \quad\text{and}\quad
\sigma \coloneqq h_\nu(\xi) \in S_\nu
\]
than on the values $k$ and $\xi$ in the parameter domain. We therefore rewrite the Fourier transformed measurements as a function of $\eta$ and $\sigma$.

\begin{definition}[Reduced measurements]
	We define the \emph{reduced measurements}
	\begin{equation}\label{eq:data_o}
		\hat{m}(h_{e_d}(k),h_{\nu}(\xi)) \coloneqq \frac{\mathcal F_{e_d^\perp+Le_d,\nu^\perp} m(k,\xi)}{ C(k,\xi)}, \qquad k\in\mathcal B^{d}_{k_0}\cap e_d^\perp, \ \xi\in\mathcal{B}^d_{k_0}\cap \nu^\perp,
	\end{equation}
	where the prefactor $C({k},\xi)$ is given by \autoref{eq:const}.
\end{definition}

Hereby, we can express the points $h_{-\nu}(\xi)$ on the right-hand side of \autoref{eq:PFDT} with the Householder transform, defined by  
\[
H_v\colon\R^d\to \R^d,\qquad H_v(x)  \coloneqq  x-2\inner{x}{v}v,
\]
for an arbitrary vector $v\in\Sp^{d-1}$, which reflects a point $x$ across the hyperplane orthogonal to $v$, via
\[
H_\nu h_\nu(\xi) = h_{-\nu}(\xi) \in S_{-\nu}
\]
as a linear function of $\sigma$, see \autoref{subfig:full}.

The reduced measurements defined in \autoref{eq:data_o} are then related to the Fourier coefficients of the scattering potential via
\begin{equation}\label{eq:fdt_nopar}
	\hat{m}(\eta,\sigma) = a(\sigma)\,\mathcal{F}_d f(\eta-\sigma) + a(H_\nu\sigma)\,\mathcal{F}_d f(\eta-H_\nu\sigma)
	\qquad \text{for all }  \eta\in S_{e_d}, \ \sigma\in S_\nu.
\end{equation}

Because of the symmetry of the right-hand side, let us, for convenience, extend the reduced measurement function $\hat m$ from \autoref{eq:data_o} to $S_{e_d}\times\Sp_{k_0}^{d-1}$ via
\begin{equation}\label{eq:data_o_ext}
\hat m(\eta,\sigma)\coloneqq\hat m(\eta,H_\nu\sigma)\qquad\text{for every } \eta\in S_{e_d}\text{ and }\sigma\in S_{-\nu}.
\end{equation}

\autoref{thm:FDT} holds for every choice of Herglotz density $a\in L^2(\Sp^{d-1}_{k_0})$. In our beam scanning setup shown in \autoref{subfig:genset}, however, the incident beam propagates in a fixed direction $\omega\in\Sp^{d-1}$, meaning that we impose \autoref{eq:halfsupport}, which restricts the Herglotz density to functions supported on the corresponding hemisphere $S_{\omega}\subseteq \Sp^{d-1}_{k_0}$. Even more, motivated by our prime example of a Gaussian beam, we want to enforce that our incident beam contains all plane waves moving in a direction of $S_\omega$, that is, $a(\sigma)\neq 0$ if and only if $\sigma\in S_\omega$.

Under this restriction, the coefficients $a(\sigma)$ and $a(H_\nu\sigma)$ in \autoref{eq:fdt_nopar} vanish if and only if $\sigma\in S_\omega$ and $H_\nu\sigma\in S_\omega$, respectively. Using the symmetry of the equation, we thus only consider values of $\sigma$ in \autoref{eq:fdt_nopar} where the first coefficient is non-zero and and differ between the two cases where the second coefficient is zero or non-zero. That is, we take $\sigma\in S_\omega$ and split $S_\omega$ into the two disjoint parts
\begin{equation}\label{eq:Sigma}
	 \Sigma_1\coloneqq \left\lbrace\sigma\in S_{\omega} \mid H_\nu\sigma\notin S_{\omega}\right\rbrace\quad\text{and}\quad \Sigma_2 \coloneqq  \left\lbrace \sigma\in S_{\omega} \mid H_\nu\sigma\in S_{\omega}\right\rbrace.
\end{equation}
A visualization of these sets is provided in \autoref{subfig:half}.

\begin{example}[The sets $\Sigma_1$ and $\Sigma_2$ for perpendicular and parallel scans]\label{ex:Sigmas}
	In the perpendicular scan configuration, where $\nu = \pm \omega$, the beam is orthogonal to the scan hyperplane. In this case, the reflection $H_\nu\sigma$ lies for every $\sigma \in S_{\omega}$ entirely outside $S_{\omega}$ so that 
	\[\Sigma_1 = S_{\omega}, \qquad \Sigma_2 = \emptyset.\] 
	In contrast, for the parallel scan, where $\omega \perp \nu$, the beam propagates along the scan plane. Here, 
	\[\Sigma_1 = \emptyset, \qquad \Sigma_2 = S_{\omega}.\] 
	For the tilted scan, where neither perpendicular nor parallel alignment holds, both sets $\Sigma_1$ and $\Sigma_2$ are non-empty, proper subsets of $S_\omega$. 
\end{example}

With this notation, we can now explicitly characterize which Fourier coefficients of the measurement $m$ correspond to only one Fourier coefficient of the scattering potential $f$ and which are a linear combination of two Fourier coefficients of $f$.

\begin{corollary}[Fourier diffraction relation for scanning measurements]\label{cor:FDT2}
	Let the assumptions of \autoref{thm:FDT} hold and assume that $a\in L^2(\Sp^{d-1}_{k_0})$ fulfills \autoref{eq:halfsupport}. Moreover, we partition the hemisphere $S_\omega$ into the sets $\Sigma_1$ and $\Sigma_2$ introduced in \autoref{eq:Sigma}.

 Then, the reduced measurement function $\hat m\colon S_{e_d}\times\Sp_{k_0}^{d-1}\to\C$ from \autoref{eq:data_o}, extended via \autoref{eq:data_o_ext}, is given by
		\begin{equation}\label{eq:FDT_o}
				\hat{m}(\eta, \sigma) =
					\begin{cases}
							a(\sigma) \mathcal{F}_df( \eta - \sigma) & \text{if } \sigma \in \Sigma_1, \vspace{2mm} \\
							a( \sigma) \mathcal{F}_df( \eta - \sigma) + a(H_\nu\sigma) \mathcal{F}_df( \eta - H_\nu\sigma) & \text{if } \sigma \in \Sigma_2
						\end{cases}
				\end{equation}
			for all $\eta \in S_{{e}_d}$.

If we additionally have that $a(\sigma)\ne 0$ if $\sigma\in S_{\omega}$, then all the coefficients $a(\sigma)$ and $a(H_\nu\sigma)$ in \autoref{eq:FDT_o} are non-zero.
\end{corollary}

	\section{Fourier reconstruction}\label{sec:Fourier_recon}

The aim is to solve the inverse problem of recovering the scattering potential $f$ from the given reduced measurements $\hat m$, where we will assume throughout that the Herglotz density $a\in L^2(\Sp_{k_0}^{d-1})$ is chosen such that $a(\sigma)\ne0$ if and only if $\sigma\in S_\omega$ so that no coefficients in \autoref{eq:FDT_o} vanish.

\autoref{cor:FDT2} provides us now with a linear equation system for the Fourier coefficients of~$f$. Besides those coefficients appearing in the first case of \autoref{eq:FDT_o}, which can be read off directly from the measurements, it is not a priori clear if the system allows us to uniquely determine all the involved Fourier coefficients.

Either way, the measurements only depend on a subset of the full Fourier data so that we have to perform a Fourier inversion using just the frequencies that are accessible from the data. This is commonly known as \emph{filtered backpropagation} \cite{Dev82}. This approach is widely used and has been intensively studied for different experimental setups. For detailed discussions on the theoretical foundations, numerical aspects, and practical implementation, we refer to \cite{Dev82, KakSla01, KirQueRitSchSet21, KirQueSet25, MueSchuGuc15_report}.

In the following, we distinguish two different backpropagation approaches based on which regions of Fourier space are used for the inversion: only those which we get directly from the first part in \autoref{eq:FDT_o} versus the maximal set of Fourier coefficients which are uniquely determined by \autoref{eq:FDT_o}.

\subsection{Naive backpropagation}
Reconstructing $f$ via Fourier inversion requires an explicit representation of the Fourier coefficients of $f$ in terms of the reduced measurements $\hat m$. According to \autoref{cor:FDT2}, the reduced measurements are influenced by the Fourier transform of $f$ on the set
\begin{equation}\label{eq:Y}
	\mathcal{Y}\coloneqq \left\lbrace \eta-\sigma \mid \eta\in S_{e_d}, \ \sigma\in S_\omega\right\rbrace \subseteq\R^d.
\end{equation}
We consider two (not necessarily disjoint) subsets $\mathcal Y_1$ and $\mathcal Y_2$ of $\mathcal Y$ corresponding to the two cases in \autoref{eq:FDT_o}.
\begin{enumerate}
\item
 On the first subset
\begin{equation}\label{eq:Y1}
	\mathcal{Y}_1\coloneqq\{\eta-\sigma \mid \eta\in S_{e_d},\ \sigma\in\Sigma_1\} \subseteq\mathcal{Y},
\end{equation}
the Fourier coefficients appear explicitly in the Fourier diffraction relation.
More precisely, since $a(\sigma)\neq 0$ for all $\sigma\in\Sigma_1$, each coefficient can be explicitly recovered via
\begin{equation}\label{eq:recd2}
	\mathcal{F}_df(\eta-\sigma)=\frac{\hat m(\eta,\sigma)}{a(\sigma)}\qquad \text{for all } \eta\in S_{e_d}, \ \sigma\in\Sigma_1.
\end{equation}

\item
The situation is different for Fourier coefficients on the second subset
\begin{equation}\label{eq:Y2}
	\mathcal{Y}_2\coloneqq\left\lbrace \eta-\sigma \mid \eta\in S_{e_d}, \ \sigma\in\Sigma_2\right\rbrace \subseteq\mathcal Y,
\end{equation}
where every measurement value is given as a linear combination of two Fourier coefficients. Consequently, reconstruction on $\mathcal{Y}_2\setminus\mathcal{Y}_1$ is not straightforward and requires additional considerations.
\end{enumerate}

A naive approach consists of ignoring this second relation in \autoref{cor:FDT2} entirely, effectively neglecting the region $\mathcal{Y}_2\setminus\mathcal{Y}_1$. Since, in this case, our knowledge of $\mathcal{F}_df$ is restricted to $\mathcal{Y}_1\subseteq\R^d$, we must extend $\mathcal{F}_df$ beyond $\mathcal{Y}_1$ in order to perform inverse Fourier transform.

Because $f$ has compact support, the Paley--Wiener theorem implies that  $\mathcal{F}_df$ is analytic and can therefore, in principle, be uniquely determined everywhere by analytic continuation if we know it on any domain. However, this theoretical argument has little practical relevance as one has only discrete data. In the absence of a better choice, we could simply extend $\mathcal{F}_df$ by zero outside $\mathcal{Y}_1$, leading to something like a low-pass filtered approximation $f^{\mathrm{naive}}$ of $f$. 
\begin{definition}[Naive backpropagation]\label{def:Nbp}
		Given the Fourier coefficients $\mathcal{F}_df$ on the set $\mathcal{Y}_1\subseteq \R^d$, we define the \emph{naive backpropagation} in our setting as 
	\begin{equation*}
		f^{\mathrm{naive}}(x)\coloneqq  (2\pi)^{-\frac{d}{2}} \int_{\R^{d}} \bm 1_{\mathcal{Y}_1}(y)\mathcal{F}_df(y)e^{i\inner{x}{y}}dy, \qquad x\in \R^d,
	\end{equation*}
	where $\bm 1_{\mathcal{Y}_1}$ denotes the indicator function of the set $\mathcal{Y}_1$.
\end{definition}

\begin{example}[Naive backpropagation for perpendicular scans]
	Consider again the perpendicular scan configuration discussed in \autoref{ex:Sigmas}, where the beam is shifted perpendicular to its propagation direction. In this case, we have $\Sigma_2=\emptyset$ and $\Sigma_1=S_{\omega}$, where $S_{\omega}$ refers to the hemisphere oriented along $\omega\in\Sp^{d-1}$. Consequently, the Fourier diffraction relation in \autoref{eq:FDT} reduces to its first equation, and the reduced measurements provide direct access to all Fourier coefficients of $f$ on the whole region $\mathcal{Y}_1 =\mathcal Y$. 
\end{example}

\subsection{Advanced backpropagation}\label{subsec:adv}
If the scan direction deviates from the perpendicular configuration, i.e., for scans where $\omega\neq \nu$, the set $\Sigma_2$ becomes non-empty, and correspondingly $\mathcal{Y}_2\neq\emptyset$. Ignoring the reduced measurements associated with $\mathcal{Y}_2\setminus\mathcal Y_1$ can then lead to a substantial loss in Fourier coverage. In the extreme case of a parallel scan, recall \autoref{ex:Sigmas}, one has $\Sigma_1 =\emptyset$ while $\Sigma_2= S_\omega$, so that this approach would not leave us any data from which we could recover the function $f$. 

We are therefore interested in investigating the following central question: is it possible to uniquely recover the individual Fourier coefficients on $\mathcal{Y}_2\setminus\mathcal Y_1$ from \autoref{eq:FDT_o}, neglecting the analytic nature of $\mathcal F_df$? That is, given the reduced measurements~$\hat m$ from \autoref{eq:data_o}, what is the largest set $Y\subseteq\R^n$ for which every solution $\phi\colon\R^d\to\C$ of the linear system
\begin{equation}\label{eq:sys_eq}
	\hat{m}(\eta,\sigma) =	\begin{cases}
		a(\sigma ) \phi( \eta - \sigma ) & \text{if } \sigma \in \Sigma_1, \vspace{2mm} \\
		a( \sigma ) \phi( \eta - \sigma ) + a(H_\nu\sigma) \phi( \eta - H_\nu\sigma ) & \text{if } \sigma \in \Sigma_2
	\end{cases} \qquad \text{for all }  \eta\in S_{e_d}
\end{equation}
fulfills that $\phi(y)=\mathcal F_df(y)$ for all $y\in Y$?

This question is discussed in detail in \cite{ElbNau26_preprint}. It turns out that the outcome depends on the dimensionality of the acquired data in relation to the dimensionality of the unknown function: In our experimental setup, the incident beam is translated along the $(d-1)$-dimensional hyperplane $\nu^\perp$, and measurements are collected on the $(d-1)$-dimensional plane $e_d^\perp+Le_d$. Thus, the acquisition process yields a $(2d-2)$-dimensional dataset, while the goal is to recover the unknown function $\mathcal F_df$ on $\R^d$. For $d>2$, the data dimension exceeds the object dimension $d$ and this additional information allows under some generic conditions that all Fourier coefficients on the set $\mathcal{Y}$ can be uniquely determined from \autoref{eq:sys_eq}. In contrast, the data and object dimension coincide for $d=2$ and the uniqueness can only be established on a specific subset $\mathcal Y_1\cup\tilde{\mathcal Y}$ of $\mathcal Y$ for some set $\tilde{\mathcal{Y}}\subseteq \mathcal{Y}_2\setminus\mathcal Y_1$.

According to \cite[Theorem 6.5]{ElbNau26_preprint}, we obtain in two dimensions the following characterization.

\begin{theorem}[Fourier coverage in two dimensions]\label{thm:uq1}
	Let $d=2$, $\omega,\nu \in \Sp^{d-1}$, $a\colon\Sp_{k_0}^{1}\to \C$ be a function satisfying $a(\sigma)\ne 0$ if and only if $\sigma\in S_{\omega}$, and $\hat m\colon S_{e_d}\times S_\omega\to\C$ be such that \autoref{eq:sys_eq} has a solution $\phi\colon\R^2\to\C$.

Define the subset
	\begin{equation}\label{eq:tildeSig}
		\tilde\Sigma\coloneqq\left\lbrace \sigma\in \Sigma_2\cap S_{e_2} \mid H_\nu\sigma \notin S_{e_2} \right\rbrace \subseteq \Sigma_2,
	\end{equation}
	and let the corresponding Fourier space coverage be
	\begin{equation}\label{eq:tildeY}
		\tilde{\mathcal{Y}}\coloneqq \left\lbrace \eta -\sigma\mid \eta\in (-\Sigma_1)\cap S_{e_2},\ \sigma\in\tilde \Sigma \right\rbrace \subseteq \mathcal{Y}_2\setminus \mathcal{Y}_1.
	\end{equation}

\begin{enumerate}
\item
We then have for every other solution $\tilde{\phi}\colon\R^2\to\C$ of \autoref{eq:sys_eq} that
\[ \phi=\tilde\phi \quad \text{almost everywhere on} \ \mathcal{Y}_1\cup \tilde{\mathcal{Y}}. \]
\item
Conversely, there exists for almost every $y\in\mathcal{Y}_2\setminus(\mathcal{Y}_1\cup \tilde{\mathcal{Y}})$ a solution $\tilde{\phi}$ of \autoref{eq:sys_eq} with $\phi(y)\neq \tilde{\phi}(y)$.
\end{enumerate}
\end{theorem}

Let us briefly illustrate these additional sets $\tilde\Sigma$ and $\tilde{\mathcal Y}$ in an example.

\begin{example}[Shape of the additional Fourier coverage in two dimensions]\label{ex:Ytilde}
	Let $d=2$ and consider the configuration $\nu\coloneqq e_2$ and $\omega\coloneq\frac1{\sqrt2}(1,-1)$ describing the scan plane and the beam direction.

	For every $\sigma=(\sigma_1,\sigma_2)\in\Sp_{k_0}^1$, the reflection at $\nu^\perp$ is given by
	\[
	H_\nu\sigma = \sigma-2\inner{\sigma}{e_2}e_2 = (\sigma_1,-\sigma_2).
	\]
	 To determine whether $\sigma\in \Sp_{k_0}^1$ belongs to $\Sigma_1$ or $\Sigma_2$, we need to check whether $\sigma\in S_{\omega}$ and the reflected vector $H_\nu\sigma$ remains in $S_\omega$, see \autoref{eq:Sigma}. Specifically, 
	\[
	\sigma\in S_\omega \ \Leftrightarrow \ \inner{\sigma}{\omega}>0\qquad \text{and} \qquad H_\nu\sigma\in S_\omega\ \Leftrightarrow \ \inner{\sigma}{H_\nu\omega}=\inner{H_\nu\sigma}{\omega} >0.
	\]
	Hence, we obtain
	\[
	\Sigma_1 = \Sp_{k_0}^1 \cap  \Omega_\omega \cap \overline{\Omega_{-H_\nu\omega}} \quad \text{and} \quad	\Sigma_2 = \Sp_{k_0}^1 \cap   \Omega_\omega \cap   \Omega_{H_\nu\omega}.
	\]
	where $\Omega_v$ denotes the half-space in the direction of the vector $v\in \Sp^{1}$, see \autoref{eq:H}.
	That is, $\Sigma_1$ and $\Sigma_2$ are the circular arcs obtained by intersecting the circle $\Sp_{k_0}^1$ with the two half-spaces defined by $\omega$ and $\mp H_\nu\omega$. In particular, we have for this configuration
	\[
	(-\Sigma_1)\cap S_{e_2}= -\Sigma_1\quad\text{and}\quad\tilde{\Sigma} = \Sigma_2\cap\overline{S_{e_2}}.
	\] 
	A visualization of this construction is presented in \autoref{fig:tildeSigma}.
	Thus, besides in the set $\mathcal Y_1$, we can recover the Fourier coefficients also almost everywhere in the additional region 
	\[
	\tilde{\mathcal{Y}} =\left\lbrace\eta-\sigma \mid \eta\in-\Sigma_1 , \ \sigma \in \Sigma_2\cap\overline{S_{e_2}}\right\rbrace \subseteq \mathcal{Y}_2\setminus \mathcal{Y}_1
	\] 
	from the data. \autoref{fig:oblique_tilted} visualizes $\tilde{\mathcal{Y}}$ for this configuration; a detailed discussion of the Fourier coverage follows in the next section.	
\end{example}

\begin{figure}[t]
	\centering	
	\includegraphics{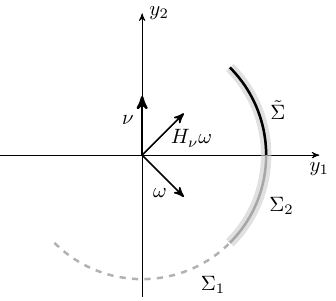}
	\caption{Illustration of the decomposition of $S_\omega\subseteq \R^2$ corresponding to \autoref{ex:Ytilde}.
		The two half-planes defined by the directions $\omega = (1/\sqrt{2}, -1/\sqrt{2})$ and $H_\nu\omega = (1/\sqrt{2}, 1/\sqrt{2})$ divide the semicircle $S_\omega$ into two arcs, $\Sigma_1$ and $\Sigma_2$.
		Because the reflection at the line $\nu^\perp$ (here, the $y_1$-axis) maps every point in $\Sigma_2$ to the opposite vertical half-plane, the part of $\Sigma_2$ lying above the $y_1$-axis constitutes $\tilde{\Sigma}$.}
	\label{fig:tildeSigma}
\end{figure}

To formulate the uniqueness result in higher dimensions $d>2$, we first introduce some notation. 
 Throughout, we denote by $\mathrm Db(\sigma)(v)$ the covariant derivative of a differentiable function $b\in C^1(\Sp^{d-1}_{k_0};\C)$ at a point $\sigma\in\Sp^{d-1}_{k_0}$ in the tangent direction $v\in\mathrm T_\sigma\Sp^{d-1}_{k_0}$ and by $\nabla b(\sigma)\in\mathrm T_\sigma\Sp^{d-1}_{k_0}$ the corresponding gradient. 
 
 In addition, for an arbitrary vector $q\in\R^d\setminus\{0\}$, we denote by
\begin{equation}\label{eq:projection}
	\pi_q\colon\R^d\to\R^d,\qquad\pi_q x\coloneqq x-\inner xq\frac{q}{\norm q^2}
\end{equation}
the orthogonal projection onto $q^\perp$.

The following result is a direct consequence of  \cite[Theorem\ 5.7]{ElbNau26_preprint} together with \cite[Theorem\ 4.5]{ElbNau26_preprint}.
\begin{theorem}[Fourier coverage in more than two dimensions]\label{thm:uq2}
Let $d\in\N\setminus\{1,2\}$ and $\omega,\nu \in \Sp^{d-1}$. Suppose that $\phi,\tilde{\phi}\in C(\R^d;\C)$ are two continuous solutions of \autoref{eq:sys_eq} for some given function $\hat m\colon S_{e_d}\times S_\omega\to\C$ and a function $a\colon\Sp^{d-1}_{k_0}\to\C$ satisfying $a(\sigma)\neq0$ if and only if  $\sigma\in S_{\omega}$.

We define the function
\begin{equation}\label{eq:def_b}
b(\sigma)\coloneqq \frac{a(\sigma)}{a(H_\nu\sigma)},\qquad \sigma\in\Sigma_2
\end{equation}
and assume that
\begin{enumerate}
	\item if $d=3$, it holds that $b\in C^2(\Sigma_2;\C)$ and $\Sigma_2 = \overline{\left\lbrace \sigma\in\Sigma_2 \mid\mathrm  D b(\sigma)(\nu\times\sigma)\neq 0\right\rbrace }$ and
	\item if $d>3$, it holds that $b\in C^1(\Sigma_2;\C)$ and $\Sigma_2 = \overline{\{\sigma\in\Sigma_2 \mid \nabla b(\sigma)\notin \C \pi_\sigma \nu\}}$.
\end{enumerate}
Then $\phi(y)=\tilde{\phi}(y)$ for all $y\in\mathcal Y$.
\end{theorem}

We remark that, in the three-dimensional case, the density condition on $\Sigma_2$ taken from \cite[Theorem\ 5.7]{ElbNau26_preprint} is formulated as a condition on the gradient, as it is done for $d>3$. The proof shows that, in three dimensions, the two formulations are equivalent. Here, we use the formulation in terms of the directional derivative, since it is more convenient for the example below.

The additional conditions in \autoref{thm:uq2} on the Herglotz density $a\in L^2(\Sp^{d-1}_{k_0})$ are not very restrictive and generally satisfied for most practical choices. We verify that the standard Gaussian beam profile given in \autoref{eq:gaussian} modeling a focused beam satisfies these conditions in the three-dimensional case. 

\begin{example}[Maximal Fourier coverage for Gaussian beams]\label{ex:checkGaus}
Let $d=3$, $\omega\in\Sp^{2}$, and $\nu\in \Sp^{2}$ be given. We consider as Herglotz density the function $a$ defined in \autoref{eq:gaussian} for some $A>0$. The function $b\colon\Sigma_2\to\C$ from \autoref{eq:def_b} is then given by
\[ b(\sigma) = \exp\left(A(\norm{\pi_\omega H_\nu\sigma}^2-\norm{\pi_\omega\sigma}^2)\right). \]

To check the condition in \autoref{thm:uq2}, we first calculate the covariant derivative of $b$ and find
\[ \mathrm  D b(\sigma)(\nu\times\sigma) = 2Ab(\sigma)\inner{\pi_\omega(H_\nu\sigma-\sigma)}{\nu\times\sigma}\quad\text{for every }\sigma\in\Sigma_2. \]
Using $H_\nu\sigma-\sigma = -2\inner{\sigma}{\nu}\nu$ gives us
\[ \mathrm  D b(\sigma)(\nu\times\sigma) = 4Ab(\sigma)\inner{\sigma}{\nu}\inner{\nu}{\omega}\inner{\omega}{\nu\times\sigma}. \]

	The derivative vanishes precisely when one of the three inner products is zero. More specifically, if
	\begin{enumerate}
		\item $\inner{\nu}{\omega} = 0$, corresponding to a scan direction parallel to the beam axis,
		\item $\inner{\sigma}{\nu}=0$, i.e., $\sigma\in\Sigma_2\cap\nu^\perp$, or
		\item $\inner{\omega}{\nu\times\sigma}=0$, which is equivalent to $\sigma\in \Sigma_2\cap\operatorname{span}\{\nu,\omega\}$.
	\end{enumerate}
	Case~(i) must be excluded, since in this configuration the derivative vanishes everywhere.
	In cases~(ii) and~(iii), the zeros occur only along one-dimensional curves on $\Sigma_2$, so they do not occupy any open part of the surface. 
	
	Consequently,
	\[
	\Sigma_2=\overline{\{\sigma\in\Sigma_2\mid \mathrm  D b(\sigma)(\nu\times\sigma)\neq0\}},
	\]
	and the assumption of \autoref{thm:uq1} is fulfilled whenever the scan normal $\nu$ and the beam direction $\omega$ satisfy $\inner{\nu}{\omega}\neq0$.
\end{example}

In summary, the naive backpropagation, introduced in \autoref{def:Nbp}, ignores Fourier data arising from the second term in the diffraction relation, thereby discarding potentially valuable information about the object in Fourier space. In two dimensions, \autoref{ex:Ytilde}, however, identifies a scan configuration that allows us to recover the additional Fourier region $\tilde{\mathcal{Y}}\neq\emptyset$. In three dimensions, \autoref{ex:checkGaus} shows that the use of focused beams allows us to achieve the full coverage $\mathcal{Y}$ defined in \autoref{eq:Y}, provided that the scan normal satisfies $\inner{\nu}{\omega}\neq 0$.

\begin{definition}[Advanced Fourier coverage]\label{def:Abp}
	Given the Fourier coefficients $\mathcal{F}_df$ on the \emph{advanced Fourier coverage} 
	\[
	\mathcal{Y}^\mathrm{adv} \coloneqq
	\begin{cases}
		\mathcal{Y}_1 \cup \tilde{\mathcal{Y}}, & d = 2,\\[4pt]
		\mathcal{Y}, & d > 2,
	\end{cases}
	\]
	the corresponding \emph{advanced backpropagation} is then defined by
	\begin{equation*}
		f^{\mathrm{adv}}(x) \coloneqq (2\pi)^{-\frac{d}{2}} \int_{\R^d} \bm 1_{\mathcal{Y}^\mathrm{adv}}(y) \, \mathcal{F}_d f(y) \, e^{i\inner{x}{y} } \, dy, \qquad x\in \R^d.
	\end{equation*}
\end{definition}

Hence, this advanced backpropagation yields an approximation $f\approx f^\mathrm{adv}$. In contrast to the naive backpropagation, the additional frequencies in $\tilde{\mathcal{Y}}\subseteq \mathcal{Y}_2$ for $d=2$ and $\mathcal{Y}_2\setminus\mathcal{Y}_1$ for $d>2$ are included by taking into account the second relation in \autoref{cor:FDT2}.

\section{Fourier coverages for 2D scan geometries}\label{sec:2Dcoverages}
Since our analysis in the special case $d=2$ has shown that, depending on the scan configuration, certain Fourier coefficients may be inaccessible, we now turn to investigating the advanced Fourier coverage $\mathcal{Y}^\mathrm{adv} = \mathcal{Y}_1\cup\tilde{\mathcal{Y}}$ for exemplary scan configurations. The influence of different Fourier coverages on the actual reconstruction of the scattering potential has been studied extensively in earlier works, including \cite{Dev89b, KakSla01, KirQueSet25}, and we therefore restrict our attention here to the analysis of the coverage itself.

Specifically, we focus on our three representative scenarios: transmission imaging, characterized by $\omega = e_2$, reflection imaging with $\omega=-e_2$, and oblique imaging, where $\omega\in \Sp^1\setminus \{\pm e_2\}$. In addition, we distinguish whether the beam is shifted perpendicular to its propagation direction or at a tilted angle, which we refer to as a perpendicular scan and a tilted scan, respectively, see \autoref{fig:scangeo}. In \autoref{tab:scan_configurations} we provide an overview of the scan setups considered, specifying the beam direction $\omega$, the scan plane normal to $\nu$, and referencing to the figures that illustrate the corresponding Fourier coverage.

\begin{table}[t]
	\centering
	\renewcommand{\arraystretch}{1.5}
	\begin{tabular}{|p{4cm}|c|c|c|c|}
		\hline
		\textbf{Imaging Mode} & \textbf{Beam Direction} & \textbf{Scan Normal} 
		& \textbf{Coverage} & \textbf{Figure} \\
		\hline
		Standard Transmission  & $\omega=e_2$  & $\nu=\omega$    & $\mathcal Y$    & \multirow{2}{*}{\autoref{fig:coverage_TI}} \\
		Tilted-Scan Transmission & $\omega=e_2$& $\nu\neq\omega$ & $\mathcal Y_1\subsetneq\mathcal Y$ & \\
		\hline
		Standard Reflection    & $\omega=-e_2$ & $\nu=\omega$    & $\mathcal Y$    & \multirow{2}{*}{\autoref{fig:coverage_RI}} \\
		Tilted-Scan Reflection & $\omega=-e_2$ & $\nu\neq\omega$ & $\mathcal Y_1\subsetneq\mathcal Y$ & \\ 
		\hline
		Oblique Incidence & $\omega\notin \{\pm e_2\}$ & $\nu=\omega$ & $\mathcal Y$      & \autoref{fig:coverage_OI} \\
		Oblique Incidence\newline\hspace*{1em} and Tilted Scan         & $\omega\notin \{\pm e_2\}$ & $\nu\neq\omega$ & $\mathcal Y_1\cup\tilde{\mathcal Y}$ & \autoref{fig:oblique_tilted} \\
		\hline
	\end{tabular}
	\caption{Overview of exemplary 2D scan configurations and their associated Fourier coverage.}
\label{tab:scan_configurations}
\end{table}

\subsection{Transmission imaging}	
We start by considering the Fourier coverages in transmission imaging. Here, the incident beam propagates along the vertical direction $\omega=e_2$ and the detector positioned at $x_2=L$ captures only the transmitted components of the scattered waves. Since $\Sigma_2\subseteq S_{e_2}$ and since we have by definition that $H_\nu\sigma\in\Sigma_2$ for every $\sigma\in\Sigma_2$, there is no point $\sigma\in\Sigma_2$ with $H_\nu\sigma\notin S_{e_2}$, which means that $\tilde{\Sigma}=\emptyset$, see \autoref{eq:tildeSig}. Consequently, no additional coverage $\tilde{\mathcal{Y}}$ arises, and the advanced Fourier coverage is 
\[
\mathcal{Y}^\mathrm{adv}= \mathcal{Y}_1=\{\eta-\sigma\mid \eta\in S_{e_2}, \ \sigma \in \Sigma_1\}\subseteq \R^2.
\]
This coverage is illustrated in \autoref{fig:coverage_TI} for various orientations $\nu\in\Sp^1$ of the scan line $\nu^\perp$. When the scan normal is aligned with the beam, i.e., $\nu = \omega$ (standard transmission), we have $\Sigma_1 = S_{e_2}$ and $\Sigma_2=\emptyset$ resulting in the largest possible coverage.  It is formed by two disks of radius $k_0$ centered at $(\pm k_0,0)$. If the scan normal $\nu$ is no longer aligned with the beam direction $e_2$, the setup corresponds to a tilted scan in transmission imaging. In this situation, the set $\Sigma_1$ shrinks, leading to a corresponding reduction of the accessible Fourier region $\mathcal{Y}_1$.

\begin{figure}[t]
	\centering
	\begin{subfigure}{0.3\textwidth}
		\includegraphics[scale = 1]{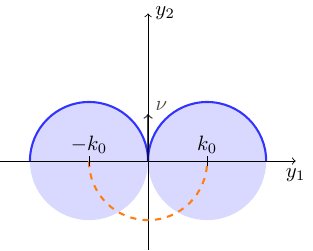}
		\caption{$\theta=\frac\pi2$}
	\end{subfigure}\hfil
	\begin{subfigure}{0.3\textwidth}
		\includegraphics[scale = 1]{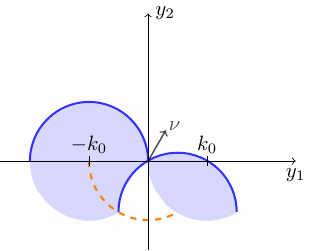}
		\caption{ $\theta=\frac{\pi}{3}$}
	\end{subfigure}\hfil
	\begin{subfigure}{0.3\textwidth}
		\includegraphics[scale = 1]{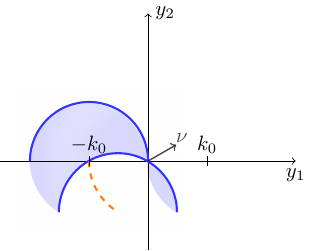}
		\caption{ $\theta=\frac{\pi}{6}$}
	\end{subfigure}
	\caption{\textbf{Perpendicular vs.\ tilted scan in transmission imaging}: the figure illustrates how the total coverage $\mathcal{Y}^\mathrm{adv}=\mathcal{Y}_1\subseteq \R^2$ changes with the scan configuration in transmission imaging. The beam propagates along $\omega = e_2$, while its focal point moves along $\nu^\perp$ with normal $\nu = (\cos\theta,\ \sin\theta) \in \Sp^1$ for different angles $\theta\in[0,\pi]$.  The corresponding set $-\Sigma_1\subseteq S_{- e_2}$ is shown as a dashed orange arc.  The perpendicular scan, $\omega=\nu=e_2$, achieves maximum coverage.}
	\label{fig:coverage_TI}
\end{figure}

\subsection{Reflection imaging}
We continue with reflection imaging, corresponding to the classical configuration used in ultrasound imaging, where a transducer both emits and receives signals. In this setup, the incident beam propagates downward, $\omega=-e_2$, and the detector records only the reflected components of the scattered waves.  As in the transmission case, we obtain $\tilde{\Sigma}=\emptyset$ and the advanced Fourier coverage reduces to $\mathcal{Y}^\mathrm{adv} =  \mathcal{Y}_1 \subseteq \R^2$. 

It is visualized in \autoref{fig:coverage_RI} for different scan orientations $\nu \in  \Sp^1$. When $\nu = \omega$, we are in standard reflection imaging mode and we have $\Sigma_1 = S_{-e_2}$ and $\Sigma_2=\emptyset$, so the region $\mathcal Y^{\mathrm{adv}}$ consists of the upper half-disk of radius $2k_0$ without the two disks of radius $k_0$ centered at $(0, \pm k_0)$. As in the transmission case, tilting the scan plane relative to the beam direction reduces $\Sigma_1$, and correspondingly the coverage $\mathcal{Y}_1$ shrinks.

\begin{figure}[t]
	\centering
	\begin{subfigure}{0.3\textwidth}
		\includegraphics[scale = 1]{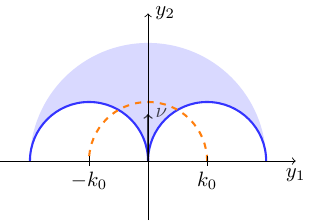}
		\caption{$\theta=\frac\pi2$}
	\end{subfigure}\hfil
	\begin{subfigure}{0.3\textwidth}
		\includegraphics[scale = 1]{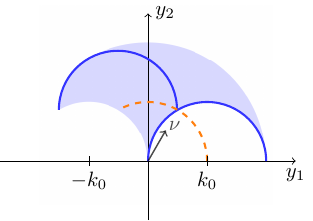}
		\caption{ $\theta=\frac{\pi}{3}$}
	\end{subfigure}\hfil
	\begin{subfigure}{0.3\textwidth}
		\includegraphics[scale = 1]{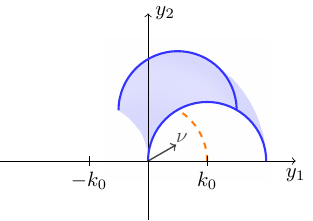}
		\caption{ $\theta=\frac{\pi}{6}$}
	\end{subfigure}
	\caption{\textbf{Perpendicular vs.\ tilted scan in reflection imaging}:  the figure illustrates how the total coverage $\mathcal{Y}^\mathrm{adv}=\mathcal{Y}_1\subseteq \R^2$ changes with the scan configuration in reflection imaging. The beam propagates along $\omega = -e_2$, while the focal point moves along $\nu^\perp$ with normal $\nu = (\cos\theta,\ \sin\theta) \in \Sp^1$. Varying the tilt angle $\theta \in [0,\pi]$ produces different accessible Fourier regions. The dashed orange arc depicts the set $-\Sigma_1 \subseteq S_{e_2}$. The perpendicular scan, $\omega = -\nu=-e_2$, again yields the largest coverage.}
	\label{fig:coverage_RI}
\end{figure}

As already discussed in \cite{Mor89, Nat15}, we observe a fundamental difference between reflection and transmission imaging. In transmission imaging, low-frequency components of the scattering potential are accessible from the measurement data, whereas in reflection imaging this is generally not the case.  This makes reflection imaging more challenging, as low-frequency components carry quantitative information about the target function. For tomographic ultrasound imaging, it is therefore natural to consider deviations from the classical reflection imaging setup, for example by sending out a beam at an oblique angle, in order to gain partial access to low-frequency information. Recent hardware developments, such as multi-aperture and flexible transducers mentioned in the introduction, support this approach and make it feasible.

\subsection{Oblique imaging}		
Motivated by this, we demonstrate coverages for oblique imaging, where the incident beam propagates in an arbitrary direction $\omega\in\Sp^1\setminus\set{\pm e_2}$. In such configurations, the detector positioned at $x_2=L$ captures a mixture of transmitted and reflected components of the scattered waves. We first discuss the case where the orientation of the scan line coincides with the beam direction, i.e., $\nu = \omega$. In this situation, no loss occurs since $\mathcal{Y}_2=\emptyset$ and the corresponding advanced coverage is given by $\mathcal{Y}^\mathrm{adv}=\mathcal{Y}_1$. These coverages for different $\omega$ are depicted in in \autoref{fig:coverage_OI}. In general, the more transmission data are available, the more low-frequency components of the object are captured.

If the scan normal is tilted against the beam direction, $\nu \neq \omega$, the relevance of our analysis in \autoref{subsec:adv} becomes evident since $\mathcal{Y}_2$ is non-empty in this case. In addition, the set $\tilde{\mathcal Y}$ may also be non-empty, allowing for an advanced Fourier coverage $\mathcal{Y}_1\cup\tilde{\mathcal{Y}}$. In \autoref{fig:oblique_tilted}, we illustrate the scenario discussed in \autoref{ex:Ytilde}, where $\omega=(1/\sqrt{2}, \ -1/\sqrt{2})$ and $\nu = e_2$. A naive backpropagation, see \autoref{def:Nbp}, reconstructs $f$ using only the frequencies in the blue region $\mathcal{Y}_1$, whereas an advanced backpropagation, see \autoref{def:Abp}, additionally includes the green region $\tilde{\mathcal{Y}}$.

\begin{figure}[t]
	\centering
	\begin{subfigure}{0.3\textwidth}
		\includegraphics[scale = 1]{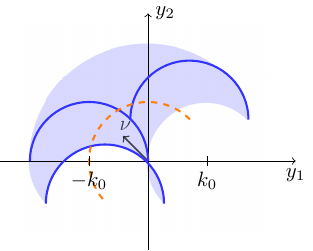}
		\caption{$\theta = -\frac{\pi}{4}$}
	\end{subfigure}\hfil
	\begin{subfigure}{0.3\textwidth}
		\includegraphics[scale = 1]{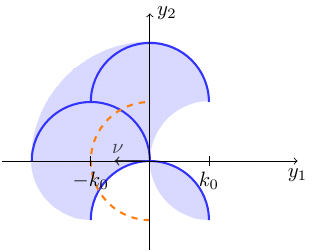}
		\caption{ $\theta = 0$}
	\end{subfigure}\hfil
	\begin{subfigure}{0.3\textwidth}
		\includegraphics[scale = 1]{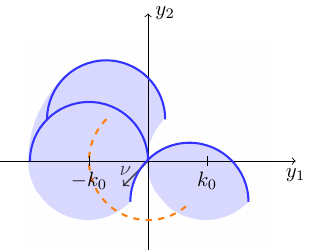}
		\caption{ $\theta =\frac{\pi}{4}$}
	\end{subfigure}
	\caption{\textbf{Perpendicular scan in oblique imaging}: visualization of the coverage $\mathcal{Y}^\mathrm{adv}=\mathcal{Y}_1\subseteq \R^2$ in  oblique imaging. The beam direction and the normal to the scan line coincide, i.e., $\nu=\omega=(\cos\theta,\ \sin\theta)$ for different angles $\theta\in[-\pi,\pi)$. Here, the set $-\Sigma_1= S_{- \omega}$ is highlighted as a dashed orange semicircle.  }
	\label{fig:coverage_OI}
\end{figure}

\begin{figure}[t]
	\centering
	\includegraphics{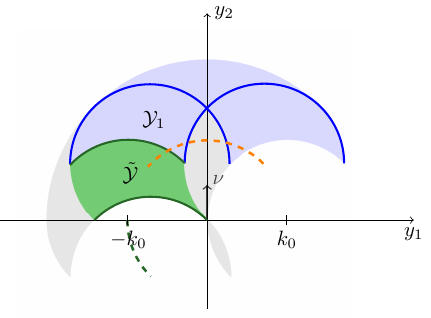}
	\caption{\textbf{Naive vs.\ advanced Fourier coverage}: visualization of the Fourier coverage for the experiment described in \autoref{ex:Ytilde}, where oblique imaging is performed with $\omega = (1/\sqrt{2},-1/\sqrt{2})$ and a tilted scan with $\nu = e_2$. The naive coverage $\mathcal{Y}_1$, shown in blue, corresponds to the Fourier coefficients accessible directly from the first equation of the Fourier diffraction relation. In contrast, the advanced coverage $\mathcal{Y}^\mathrm{adv} = \mathcal{Y}_1 \cup \tilde{\mathcal{Y}}$, with $\tilde{\mathcal{Y}}$ highlighted in green, includes additional coefficients obtained via the second part of the relation. Here, the dashed orange arcs depicts the set $-\Sigma_1$ while the dashed green arc shows $-\tilde{\Sigma}$. The gray region, $\mathcal{Y}_2 \setminus (\mathcal{Y}_1 \cup \tilde{\mathcal{Y}})$, remains inaccessible.}
	\label{fig:oblique_tilted}
\end{figure}

\section{Conclusions}\label{sec:conclusion}
In this work, we extended classical diffraction tomography to a general scanning geometry in which focused beams are emitted from one side and scanned across the object. The beam and scan directions may be chosen independently. This provides, for the first time, a framework that allows diffraction tomography to be applied directly to standard ultrasound setups while also offering flexibility for future system designs.

Our central result is a new Fourier diffraction theorem that relates these scanning measurements to the object's scattering potential in Fourier space. Depending on the scanning geometry, some Fourier coefficients can be recovered directly from the data, while we have to solve a linear equations system to obtain the others.

We discussed that in dimensions higher than two all Fourier coefficients appearing in the measurements can, in principle, be recovered, since the associated linear system is solvable in most relevant configurations. In particular, solvability is guaranteed when the Herglotz wave models a Gaussian beam, which is often the situation encountered in practice.

In two dimensions, the situation is different, as the linear system is not uniquely solvable in general. However, the theory allows us to identify exactly which Fourier coefficients can be recovered, and in some cases we can still extract quite a bit more than only the directly accessible Fourier coefficients, which can be seen in the visualizations of the resulting Fourier coverages.

We are therefore confident that this framework can benefit the reconstructions for oblique imaging setups.

	\subsection*{Acknowledgments}
	This research was funded in whole, or in part, by the Austrian Science Fund
	(FWF) 10.55776/P34981 -- New Inverse Problems of Super-Resolved Microscopy (NIPSUM)
	and SFB 10.55776/F68 ``Tomography Across the Scales'', project F6807-N36
	(Tomography with Uncertainties) and F6804-N36 (Quantitative Coupled Physics Imaging). For open access purposes, the author has
	applied a CC BY public copyright license to any author-accepted manuscript
	version arising from this submission.
	The financial support by the Austrian Federal Ministry for Digital and Economic
	Affairs, the National Foundation for Research, Technology and Development and the Christian Doppler
	Research Association is gratefully acknowledged.

	\printbibliography

@INCOLLECTION{Agm90,
  AUTHOR = {Agmon, S.},
  BOOKTITLE = {Analysis, et Cetera},
  DATE = {1990},
  DOI = {10.1016/b978-0-12-574249-8.50008-0},
  PAGES = {39--76},
  TITLE = {A Representation Theorem for Solutions of the Helmholtz Equation and Resolvent Estimates for The Laplacian},
}

@ARTICLE{AgrPat79,
  AUTHOR = {Agrawal, G. P. and Pattanayak, D. N.},
  PUBLISHER = {Optica Publishing Group},
  DATE = {1979},
  DOI = {10.1364/josa.69.000575},
  JOURNALTITLE = {Journal of the Optical Society of America},
  NUMBER = {4},
  PAGES = {575--578},
  SHORTJOURNAL = {J. Opt. Soc. Amer.},
  TITLE = {Gaussian beam propagation beyond the paraxial approximation},
  VOLUME = {69},
}

@ARTICLE{AmmChoZou16,
  AUTHOR = {Ammari, H. and Chow, Y. T. and Zou, J.},
  DATE = {2016},
  DOI = {10.1137/15m1043959},
  ISSN = {0036-1399},
  JOURNALTITLE = {{SIAM} Journal on Applied Mathematics},
  NUMBER = {3},
  PAGES = {1000--1030},
  SHORTJOURNAL = {{SIAM} J. Appl. Math.},
  TITLE = {Phased and Phaseless Domain Reconstructions in the Inverse Scattering Problem via Scattering Coefficients},
  VOLUME = {76},
}

@ARTICLE{BorDruMamZas18,
  AUTHOR = {Borcea, L. and Druskin, V. and Mamonov, A. V. and Zaslavsky, M.},
  DATE = {2018},
  DOI = {10.1088/1361-6420/aabb16},
  ISSN = {0266-5611},
  JOURNALTITLE = {Inverse Problems},
  NUMBER = {6},
  PAGES = {065008},
  SHORTJOURNAL = {Inverse Probl.},
  TITLE = {Untangling the nonlinearity in inverse scattering with data-driven reduced order models},
  VOLUME = {34},
}

@ARTICLE{BorDruMamZasZim20,
  AUTHOR = {Borcea, L. and Druskin, V. and Mamonov, A. V. and Zaslavsky, M. and Zimmerling, J.},
  DATE = {2020},
  DOI = {10.1137/19m1296355},
  ISSN = {1936-4954},
  JOURNALTITLE = {{SIAM} Journal on Imaging Sciences},
  NUMBER = {2},
  PAGES = {685--723},
  SHORTJOURNAL = {{SIAM} J. Imaging Sciences},
  TITLE = {Reduced Order Model Approach to Inverse Scattering},
  VOLUME = {13},
}

@ARTICLE{BorGarMamZim22,
  AUTHOR = {Borcea, L. and Garnier, J. and Mamonov, A. V. and Zimmerling, J.},
  DATE = {2022},
  DOI = {10.1088/1361-6420/ac41d0},
  ISSN = {0266-5611},
  JOURNALTITLE = {Inverse Problems},
  NUMBER = {2},
  PAGES = {025004},
  SHORTJOURNAL = {Inverse Probl.},
  TITLE = {Reduced order model approach for imaging with waves},
  VOLUME = {38},
}

@BOOK{CakColHad16,
  AUTHOR = {Cakoni, F. and Colton, D. and Haddar, H.},
  PUBLISHER = {Society for Industrial and Applied Mathematics},
  DATE = {2016-11},
  DOI = {10.1137/1.9781611974461},
  TITLE = {Inverse Scattering Theory and Transmission Eigenvalues},
}

@ARTICLE{ColKir96,
  AUTHOR = {Colton, D. and Kirsch, A.},
  DATE = {1996},
  DOI = {10.1088/0266-5611/12/4/003},
  ISSN = {0266-5611},
  JOURNALTITLE = {Inverse Problems},
  NUMBER = {4},
  PAGES = {383--393},
  SHORTJOURNAL = {Inverse Probl.},
  TITLE = {A simple method for solving inverse scattering problems in the resonance region},
  VOLUME = {12},
}

@ARTICLE{ColKirPai89,
  AUTHOR = {Colton, D. and Kirsch, A. and P\"{a}iv\"{a}rinta, L.},
  DATE = {1989},
  DOI = {10.1137/0520096},
  ISSN = {0036-1410},
  JOURNALTITLE = {{SIAM} Journal on Mathematical Analysis},
  NUMBER = {6},
  PAGES = {1472--1483},
  SHORTJOURNAL = {{SIAM} J. Math. Anal.},
  TITLE = {Far-Field Patterns for Acoustic Waves in an Inhomogeneous Medium},
  VOLUME = {20},
}

@BOOK{ColKre19,
  AUTHOR = {Colton, D. and Kress, R.},
  PUBLISHER = {Springer},
  DATE = {2019},
  EDITION = {4},
  ISBN = {978-3-030-30350-1},
  NUMBER = {93},
  SERIES = {Applied Mathematical Sciences},
  TITLE = {Inverse Acoustic and Electromagnetic Scattering Theory},
}

@ARTICLE{ColMon88,
  AUTHOR = {Colton, D. and Monk, P.},
  DATE = {1988},
  DOI = {10.1093/qjmam/41.1.97},
  JOURNALTITLE = {The Quarterly Journal of Mechanics and Applied Mathematics},
  NUMBER = {1},
  PAGES = {97--125},
  SHORTJOURNAL = {Q. J. Mech. and Appl. Math.},
  TITLE = {The inverse scattering problem for time-harmonic acoustic waves in an inhomogeneous medium},
  VOLUME = {41},
}

@ARTICLE{Dev82,
  AUTHOR = {Devaney, A.},
  DATE = {1982},
  DOI = {10.1016/0161-7346(82)90017-7},
  JOURNALTITLE = {Ultrasonic Imaging},
  NUMBER = {4},
  PAGES = {336--350},
  SHORTJOURNAL = {Ultrason. Imaging},
  TITLE = {A filtered backpropagation algorithm for diffraction tomography},
  VOLUME = {4},
}

@ARTICLE{Dev89b,
  AUTHOR = {Devaney, A. J.},
  DATE = {1989},
  DOI = {10.1088/0266-5611/5/4/006},
  ISSN = {0266-5611},
  JOURNALTITLE = {Inverse Problems},
  NUMBER = {4},
  PAGES = {501--521},
  SHORTJOURNAL = {Inverse Probl.},
  TITLE = {The limited-view problem in diffraction tomography},
  VOLUME = {5},
}

@REPORT{ElbNau26_preprint,
  AUTHOR = {Elbau, P. and Naujoks, N.},
  DATE = {2026},
  TITLE = {Invertibility of the Fourier Diffraction Relation in Raster Scan Diffraction Tomography},
  TYPE = {Preprint on ArXiv},
  NUMBER     = {2602.17344},
  DOI = {10.48550/arXiv.2602.17344},
}

@INPROCEEDINGS{Gre77,
  AUTHOR = {Greenleaf, J. F.},
  BOOKTITLE = {Ultrasonics Symposium Proceedings},
  DATE = {1977},
  PAGES = {989--995},
  TITLE = {Quantitative cross-sectional imaging of ultrasound parameters.},
}

@ARTICLE{GutKli93,
  AUTHOR = {Gutman, S. and Klibanov, M.},
  DATE = {1993},
  DOI = {10.1016/0895-7177(93)90076-b},
  JOURNALTITLE = {Mathematical and Computer Modelling},
  NUMBER = {1},
  PAGES = {5--31},
  SHORTJOURNAL = {Math. Comput. Modelling},
  TITLE = {Regularized Quasi-Newton method for inverse scattering problems},
  VOLUME = {18},
}

@ARTICLE{HalHooSamSchwLop24,
  AUTHOR = {van Hal, V. H. J. and de Hoop, H. and van Sambeek, M. R. H. M. and Schwab, H.-M. and Lopata, R. G. P.},
  DATE = {2024},
  DOI = {10.3389/fphys.2024.1320456},
  JOURNALTITLE = {Frontiers in Physiology},
  SHORTJOURNAL = {Front. Physiol.},
  TITLE = {In vivo bistatic dual-aperture ultrasound imaging and elastography of the abdominal aorta},
  VOLUME = {15},
}

@ARTICLE{Hoh01,
  AUTHOR = {Hohage, T.},
  DATE = {2001},
  DOI = {10.1088/0266-5611/17/6/314},
  ISSN = {0266-5611},
  JOURNALTITLE = {Inverse Problems},
  NUMBER = {6},
  PAGES = {1743--1763},
  SHORTJOURNAL = {Inverse Probl.},
  TITLE = {On the numerical solution of a three-dimensional inverse medium scattering problem},
  VOLUME = {17},
}

@ARTICLE{HohNovSiv24,
  AUTHOR = {Hohage, T. and Novikov, R. G. and Sivkin, V. N.},
  DATE = {2024},
  DOI = {10.1088/1361-6420/ad6fc6},
  ISSN = {0266-5611},
  JOURNALTITLE = {Inverse Problems},
  NUMBER = {10},
  PAGES = {105007},
  SHORTJOURNAL = {Inverse Probl.},
  TITLE = {Phase retrieval and phaseless inverse scattering with background information},
  VOLUME = {40},
}

@BOOK{Hos19,
  EDITOR = {Hoskins, P. R. and Martin, K. and Thrush, A.},
  DATE = {2019},
  DOI = {10.1201/9781138893603},
  TITLE = {Diagnostic Ultrasound},
}

@BOOK{KakSla01,
  AUTHOR = {Kak, A. C. and Slaney, M.},
  LOCATION = {Philadelphia, PA},
  PUBLISHER = {Society for Industrial and Applied Mathematics (SIAM)},
  DATE = {2001},
  DOI = {10.1137/1.9780898719277},
  NOTE = {Reprint of the 1988 original},
  SERIES = {Classics in Applied Mathematics},
  TITLE = {Principles of Computerized Tomographic Imaging},
  VOLUME = {33},
}

@ARTICLE{KirNauSchYan24,
  AUTHOR = {Kirisits, C. and Naujoks, N. and Scherzer, O. and Yang, H.},
  DATE = {2024},
  DOI = {10.1088/1361-6420/ad7d2d},
  ISSN = {0266-5611},
  JOURNALTITLE = {Inverse Problems},
  NUMBER = {11},
  PAGES = {115007},
  SHORTJOURNAL = {Inverse Probl.},
  TITLE = {Diffraction tomography for incident Herglotz waves},
  VOLUME = {40},
}

@ARTICLE{KirQueRitSchSet21,
  AUTHOR = {Kirisits, C. and Quellmalz, M. and Ritsch-Marte, M. and Scherzer, O. and Setterqvist, E. and Steidl, G.},
  DATE = {2021},
  DOI = {10.1088/1361-6420/ac2749},
  ISSN = {0266-5611},
  JOURNALTITLE = {Inverse Problems},
  NUMBER = {11},
  PAGES = {115002},
  SHORTJOURNAL = {Inverse Probl.},
  TITLE = {Fourier reconstruction for diffraction tomography of an object rotated into arbitrary orientations},
  VOLUME = {37},
}

@ARTICLE{KirQueSet25,
  AUTHOR = {Kirisits, C. and Quellmalz, M. and Setterqvist, E.},
  DATE = {2025},
  DOI = {10.1137/24M167370X},
  ISSN = {1936-4954},
  ISSUE = {1},
  JOURNALTITLE = {{SIAM} Journal on Imaging Sciences},
  PAGES = {665--700},
  SHORTJOURNAL = {{SIAM} J. Imaging Sciences},
  TITLE = {Generalized Fourier Diffraction Theorem and Filtered Backpropagation for Tomographic Reconstruction},
  VOLUME = {18},
}

@ARTICLE{KliRom16,
  AUTHOR = {Klibanov, M. V. and Romanov, V. G.},
  DATE = {2016},
  DOI = {10.1137/15m1022367},
  ISSN = {0036-1399},
  JOURNALTITLE = {{SIAM} Journal on Applied Mathematics},
  NUMBER = {1},
  PAGES = {178--196},
  SHORTJOURNAL = {{SIAM} J. Appl. Math.},
  TITLE = {Reconstruction Procedures for Two Inverse Scattering Problems Without the Phase Information},
  VOLUME = {76},
}

@BOOK{Kut91,
  AUTHOR = {Kuttruff, H.},
  LOCATION = {Amsterdam},
  PUBLISHER = {Elsevier Science Publishers Ltd},
  DATE = {1991},
  DOI = {10.1007/978-94-011-3846-8},
  TITLE = {Ultrasonics: Fundamentals and Applications},
}

@ARTICLE{Mor89,
  AUTHOR = {Mora, P.},
  PUBLISHER = {Society of Exploration Geophysicists},
  DATE = {1989},
  DOI = {10.1190/1.1442625},
  JOURNALTITLE = {Geophysics},
  NUMBER = {12},
  PAGES = {1575--1586},
  SHORTJOURNAL = {Geophysics},
  TITLE = {Inversion = migration $+$ tomography},
  VOLUME = {54},
}

@MISC{MueSchuGuc15_report,
  AUTHOR = {M\"{u}ller, P. and Sch\"{u}rmann, M. and Guck, J.},
  URL = {https://arxiv.org/abs/1507.00466v3},
  DATE = {2015},
  EPRINT = {1507.00466},
  TITLE = {The theory of diffraction tomography},
  TYPE = {arxiv},
}

@ARTICLE{NahPanKak84,
  AUTHOR = {Nahamoo, D. and Pan, S. X. and Kak, A. C.},
  DATE = {1984},
  DOI = {10.1109/t-su.1984.31502},
  JOURNALTITLE = {{IEEE} Transactions on Sonics and Ultrasonics},
  NUMBER = {4},
  PAGES = {218--229},
  SHORTJOURNAL = {{IEEE} Trans. Son. Ultrason.},
  TITLE = {Synthetic Aperature Diffraction Tomography and Its Interpolation-Free Computer Implementation},
  VOLUME = {31},
}

@INCOLLECTION{Nat15,
  AUTHOR = {Natterer, F.},
  EDITOR = {Scherzer, O.},
  LOCATION = {New York},
  PUBLISHER = {Springer},
  BOOKTITLE = {Handbook of Mathematical Methods in Imaging},
  DATE = {2015},
  DOI = {10.1007/978-3-642-27795-5_37-2},
  EDITION = {2},
  ISBN = {978-1-4939-0789-2},
  PAGES = {1253--1278},
  TITLE = {Sonic Imaging},
}

@ARTICLE{NeePetVerPeeHaa24,
  AUTHOR = {van Neer, P. L. M. J. and Peters, L. C. J. M. and Verbeek, R. G. F. A. and Peeters, B. and de Haas, G. and H\"{o}rchens, L. and Fillinger, L. and Schrama, T. and Merks-Swolfs, E. J. W. and Gijsbertse, K. and Saris, A. E. C. M. and Mozaffarzadeh, M. and Menssen, J. M. and de Korte, Ch. L. and van der Steen, J.-L. P. J. and Volker, A. W. F. and Gelinck, G. H.},
  DATE = {2024},
  DOI = {10.1038/s41467-024-47074-1},
  JOURNALTITLE = {Nature Communications},
  NUMBER = {1},
  SHORTJOURNAL = {Nat. Commun.},
  TITLE = {Flexible large-area ultrasound arrays for medical applications made using embossed polymer structures},
  VOLUME = {15},
}

@ARTICLE{NorLin79,
  AUTHOR = {Norton, S. J. and Linzer, M.},
  DATE = {1979},
  DOI = {10.1177/016173467900100205},
  JOURNALTITLE = {Ultrasonic Imaging},
  NUMBER = {2},
  PAGES = {154--184},
  SHORTJOURNAL = {Ultrason. Imaging},
  TITLE = {Ultrasonic Reflectivity Tomography: Reconstruction with Circular Transducer Arrays},
  VOLUME = {1},
}

@ARTICLE{Nov15,
  AUTHOR = {Novikov, R. G.},
  DATE = {2015},
  DOI = {10.1016/j.bulsci.2015.04.005},
  JOURNALTITLE = {Bulletin des Sciences Math\'{e}matiques},
  NUMBER = {8},
  PAGES = {923--936},
  SHORTJOURNAL = {Bull. Sci. Math.},
  TITLE = {Formulas for phase recovering from phaseless scattering data at fixed frequency},
  VOLUME = {139},
}

@ARTICLE{Som12,
  AUTHOR = {Sommerfeld, A.},
  DATE = {1912},
  JOURNALTITLE = {Jahresbericht der Deutschen Mathematiker-Vereinigung},
  PAGES = {309--352},
  SHORTJOURNAL = {Jber. Deutsch. Math. Verein.},
  TITLE = {Die Greensche Funktion der Schwingungslgleichung},
  VOLUME = {21},
}

@ARTICLE{Wol69,
  AUTHOR = {Wolf, E.},
  DATE = {1969},
  DOI = {10.1016/0030-4018(69)90052-2},
  JOURNALTITLE = {Optics Communications},
  NUMBER = {4},
  PAGES = {153--156},
  SHORTJOURNAL = {Opt. Commun.},
  TITLE = {Three-dimensional structure determination of semi-transparent objects from holographic data},
  VOLUME = {1},
}

	
\end{document}